\documentclass[11pt,a4paper]{amsart}
\usepackage[T1]{fontenc}
\usepackage[utf8]{inputenc}
\usepackage{mathpazo}
\usepackage{amsmath,amssymb,mathtools,mathrsfs}
\usepackage[a4paper,margin=2.8cm]{geometry}
\usepackage{microtype,enumitem}
\usepackage[hidelinks]{hyperref}
\hypersetup{pdftitle={Cauchy laws for zeta logarithmic derivatives and stationary Stieltjes transforms},pdfauthor={Joseph Najnudel and Ashkan Nikeghbali},pdfkeywords={Cauchy distribution, stationary random measures, Stieltjes transform, Riemann zeta function, finite fields}}
\allowdisplaybreaks[2]
\numberwithin{equation}{section}
\newtheorem{theorem}{Theorem}[section]
\newtheorem{proposition}[theorem]{Proposition}
\newtheorem{lemma}[theorem]{Lemma}
\newtheorem{corollary}[theorem]{Corollary}
\makeatletter
\@addtoreset{theorem}{subsection}
\makeatother
\renewcommand{\thetheorem}{\ifnum\value{subsection}>0 \thesubsection.\arabic{theorem}\else\thesection.\arabic{theorem}\fi}

\theoremstyle{definition}

\newtheorem{example}[theorem]{Example}
\theoremstyle{remark}
\newtheorem{remark}[theorem]{Remark}
\newcommand{\R}{\mathbb R}
\newcommand{\C}{\mathbb C}
\newcommand{\Z}{\mathbb Z}

\newcommand{\E}{\mathbb E}
\renewcommand{\P}{\mathbb P}
\newcommand{\I}{\mathcal I}
\newcommand{\Law}{\operatorname{Law}}
\newcommand{\supp}{\operatorname{supp}}
\newcommand{\pv}{\operatorname{pv}}
\newcommand{\dd}{\,\mathrm d}
\newcommand{\ii}{\mathrm i}
\newcommand{\Cauchy}{\operatorname{Cauchy}}
\newcommand{\dBL}{d_{\mathrm{BL}}}
\newcommand{\dK}{d_{\mathrm K}}

\newcommand{\Hol}{\mathcal H}

\title[Cauchy laws for zeta logarithmic derivatives]{Cauchy laws for zeta logarithmic derivatives\\ and stationary Stieltjes transforms}
\author{Joseph Najnudel}
\author{Ashkan Nikeghbali}
\date{5 September 2026}
\subjclass[2020]{11M26, 60G55, 60E07, 11G20, 11G25, 60B20}
\keywords{Cauchy distribution, stationary random measure, Stieltjes transform, Riemann zeta function, zeta functions over finite fields}
\begin{document}
\begin{abstract}
We prove an unconditional Cauchy limit for a symmetrically truncated Stieltjes transform of the projected ordinates of the zeros of the Riemann zeta function. A suitable normalization of the logarithmic derivative of $\zeta$ on the critical line has the same limit provided that the sum of the distances to the critical line of the zeros lying to its right, counted with multiplicity up to height 
$2T$, is $o(T)$; we give an explicit bound on the comparison error. The proof rests on a convergence theorem for the Stieltjes transform of positive stationary point measures whose counting discrepancy satisfies an integrability criterion. This theorem is obtained by comparison with the transforms of periodic point measures, combined with truncation bounds and a transfer theorem. Over finite fields, a classical cotangent identity yields Cauchy limits for the logarithmic derivatives of zeta functions of varieties: the law is exactly Cauchy for nonconstant pure cohomological factors, and Poincar\'e duality gives explicit errors for full zeta functions, including Cauchy limits for smooth hypersurfaces of increasing degree, uniformly in the base field.
\end{abstract}
\maketitle
\pagestyle{plain}
\section{Introduction}
We study Cauchy laws for logarithmic derivatives, with applications to the Riemann zeta function and to zeta functions over finite fields. We begin with the two number-field limits that motivate the approximation results of this paper.

For $T \geq 100$, put $L=\log T$ and let $t_T$ be uniform on $[T/L,T]$. Write the distinct nontrivial zeros of $\zeta$ as $\rho=\beta+\ii\gamma$, with multiplicities $m(\rho)$, and set
\begin{equation}\label{eq:intro-horizontal}
 u_{\rho,T}=\frac{(\gamma-t_T)L}{2\pi},\qquad
 H(T)=\sum_{\substack{0<\gamma\le2T\\\beta>1/2}}
 m(\rho)(\beta-1/2).
\end{equation}
We write $\Cauchy(a,b)$ for the Cauchy law with location $a$ and scale $b>0$, 
i.e. the law of $a + bX$, where $X$ is a random variable following the standard Cauchy distribution $\Cauchy(0,1)$, which has density $[\pi(1+x^2)]^{-1}$.

\begin{theorem}[Main arithmetic results]\label{thm:intro-arithmetic}
Without assuming the Riemann hypothesis,
\begin{equation}\label{eq:intro-projected}
 P_T:=-\frac1\pi\sum_{0<|u_{\rho,T}|\le L^2}
 \frac{m(\rho)}{u_{\rho,T}}
 \ \Longrightarrow\ \Cauchy(0,1).
\end{equation}
If $H(T)=o(T)$, then
\begin{equation}\label{eq:intro-zeta}
 A_T:=\frac{2\ii}{L}\frac{\zeta'}{\zeta}
 \left(\frac12+\ii t_T\right)+\ii
 \ \Longrightarrow\ \Cauchy(0,1).
\end{equation}
The second convergence also holds with $t_T$ replaced by $T\omega$, where $\omega$ is uniform on $[0,1]$.
\end{theorem}

These assertions are proved in Theorems~\ref{thm:projected} and~\ref{thm:zeta}. All zero counts include multiplicities, and no simplicity hypothesis is used.

\subsection{Horizontal collapse and the comparison estimate}
Horizontal projection of the renormalized zeros of $\zeta$ gives a positive point measure with stationary local limits. Counting estimates identify the Cauchy law of its truncated transform $P_T$ without requiring a unique limiting point process. For the comparison with the renormalized logarithmic derivative $A_T$ of $\zeta$, we prove
\begin{equation}\label{eq:intro-comparison}
 \dBL\bigl(\Law(A_T),\Law(P_T)\bigr)
 \le C\sqrt{\frac{H(T)}T}
       +\frac{C}{\log T},
\end{equation}
where $\dBL$ is the bounded-Lipschitz distance on $\C$. A square-root moment controls the off-critical correction; averaging the local factorization error gives the last term. A rate of convergence to the Cauchy law would require a rate for $P_T$. 

By Selberg's estimate, $H(T)=o(T)$ is equivalent to microscopic horizontal collapse: for every $\varepsilon>0$, only $o(T\log T)$ zeros up to height $2T$ have horizontal distance at least $\varepsilon/\log T$ from the critical line. Proposition~\ref{prop:collapse-equivalence} also characterizes it by the vanishing of every normalized positive horizontal moment. It follows from density one on the critical line, and permits small off-critical displacements.

Under the Riemann hypothesis, Sodin \cite[Section~3.4.5, equation~(55)]{Sodin} stated the logarithmic-derivative limit using Aizenman--Warzel theory. Here we prove the unconditional projected limit and the comparison under the stated weaker hypotheses. A critical-line local limit with Lebesgue mean measure suffices, including sine-process convergence. We also identify the Cauchy coordinate in a stochastic-zeta field limit.

\subsection{A quantitative probabilistic construction}
We approximate positive stationary point measures $M$ by periodic measures. Throughout, $M=\sum_{x\in X}w_x\delta_x$ is a positive random point measure which is locally finite. Our transform convention is
\begin{equation}\label{eq:intro-transform}
 F_M(z)=\lim_{A\to\infty}\int_{[-A,A]}\frac{\mathrm dM(x)}{z-x}.
\end{equation}
For stationary $M$ of intensity $K=\E M((0,1])\in(0,\infty)$, set $N_M(x)=M((0,x])$ for $x\ge0$ and $N_M(x)=-M((x,0])$ for $x<0$. Our criterion is
\begin{equation}\label{eq:intro-discrepancy}
 \E|N_M(x)-Kx|\le h(|x|)\quad (|x|\ge1),\qquad
 \int_1^\infty\frac{h(r)}{r^2}\dd r<\infty.
\end{equation}
Theorem~\ref{thm:stationary} below constructs $F_M$ simultaneously outside the support, proves translation invariance in distribution, and identifies $F_M(s)/(\pi K)$ as standard Cauchy, independently of the sigma-field $\I$ of events which are  invariant under all real shifts. No ergodicity is required. Truncation error is $O(A^{\alpha-1})$ for discrepancy $O(A^\alpha)$, $\alpha<1$, and $O(\sqrt{\log A}/A)$ for sine processes. 
Our proof starts by showing our result for periodic measures.
The finite input is a weighted cotangent identity of Pitman and Williams \cite{PitmanWilliams,Williams}; see \cite[Lemma~3.1]{PillaiMeng} and Boole's related identity \cite{Boole}. Aizenman--Warzel theory \cite{AW} encompasses the qualitative scalar laws. Our method 
 emphasis is periodization and quantitative transfer (Section~\ref{sec:analytic}).

\subsection{Zeta functions over finite fields}
Purity and the finite identity give exact Cauchy laws for nonconstant cohomological factors. For a smooth projective geometrically connected variety of dimension $n\ge1$ with $b_n>0$, Poincar\'e duality gives a coupling of the normalized, Euler-characteristic-centered logarithmic derivative of its full zeta function to a standard Cauchy variable, with error at most
\[
 \frac4{b_n}\sum_{j<n}
 \frac{b_jq^{(n-j)/2}}{q^{n-j}-1}.
\]
For smooth degree-$d$ hypersurfaces in fixed positive dimension $n$, this is $O_n(d^{-n-1})$, uniformly in the base field.

In Section~\ref{sec:finite}, we prove that the Stieltjes transform of a periodic positive
point measure, evaluated at a uniform point of a period, is Cauchy distributed. In
Section~\ref{sec:stationary}, we extend this result to random stationary positive point
measures, by comparison with periodic measures. In Section~\ref{sec:truncation}, we give
quantitative estimates on truncations of Stieltjes transforms and study the effect of asymmetric truncations. In Section~\ref{sec:transfer}, we prove a transfer
theorem: under suitable conditions, the Stieltjes transforms of a sequence of measures
converging to a stationary positive point measure converge to a Cauchy variable. In
Section~\ref{sec:analytic}, we compare our results with those given by the
Aizenman--Warzel theory. In Sections~\ref{sec:projected} and~\ref{sec:correction}, we apply
the results of the previous sections to the Riemann zeta function, using the local
factorization proved in Appendix~\ref{app:factorization}. Section~\ref{sec:functionfields}
treats zeta functions over finite fields.

\section{The finite and the periodic setting}\label{sec:finite}

The starting point is a weighted cotangent identity. Its conclusion holds conditionally on the entire configuration of poles and weights.

\begin{proposition}[Weighted cotangent identity]\label{prop:cotangent}
Let $\theta_1,\ldots,\theta_n\in[0,\pi)$ and $w_1,\ldots,w_n\ge0$, with $\sum_{j=1}^n w_j=1$, be possibly random. If $\Theta$ is uniform on $[0,\pi)$ and independent of these variables, then, conditionally on $(\theta_j,w_j)_{j=1}^n$,
\begin{equation}\label{eq:cotangent-law}
\sum_{j=1}^n w_j\cot(\Theta-\theta_j)\sim\Cauchy(0,1).
\end{equation}
In particular, the variable in \eqref{eq:cotangent-law} is independent of the weighted configuration.
\end{proposition}

\begin{proof}
Conditioning reduces the question to deterministic angles and weights. Discard zero weights and group equal angles, retaining $n$ for their number. After translation modulo $\pi$, assume $0=\theta_1<\cdots<\theta_n<\pi$. Put $\theta_{n+1}=\pi$ and
\[
h(\theta)=\sum_{j=1}^n w_j\cot(\theta-\theta_j).
\]
On each interval $(\theta_j,\theta_{j+1})$, the function $h$ decreases strictly from $+\infty$ to $-\infty$. For $q\in\R$, let $\theta'_j$ be its unique solution of $h(\theta'_j)=q$ in that interval. With $z=e^{2\ii\theta}$, this equation becomes
\[
\sum_{j=1}^n w_j(z+e^{2\ii\theta_j})\prod_{k\ne j}(z-e^{2\ii\theta_k})
+\ii q\prod_{j=1}^n(z-e^{2\ii\theta_j})=0.
\]
The leading coefficient is $1+\ii q$ and the constant coefficient is
$(-1)^n e^{2\ii\sum_j\theta_j}(\ii q-1)$. The $\theta'_j$ lie in disjoint subintervals of one period, so $e^{2\ii\theta'_j}$ are all its $n$ distinct roots. Their product gives
\[
e^{2\ii\sum_j(\theta'_j-\theta_j)}=\frac{\ii q-1}{\ii q+1}.
\]
The sum $\sum_j(\theta'_j-\theta_j)$ belongs to $(0,\pi)$, since each summand lies between zero and the corresponding gap. It therefore equals $\operatorname{arccot}(q)$; throughout, this inverse cotangent takes values in $(0,\pi)$. Since $\{h>q\}$ is, up to its endpoints, the union of the intervals $(\theta_j,\theta'_j)$,
\[
\P(h(\Theta)>q)=\frac{\operatorname{arccot}(q)}{\pi}.
\]
Differentiation gives the density $1/[\pi(1+q^2)]$. The conditional law does not depend on the configuration, proving the final assertion.
\end{proof}

In the following result, positive weighted point measures are understood with the locally finite support convention of the Introduction.

\begin{theorem}[Periodic measures]\label{thm:periodic}
Let $M$ be a nonzero positive weighted point measure which is $B$-periodic, where $B>0$, and set $K=B^{-1}M([0,B))$. For every $s\notin\supp M$, the symmetric principal value
\begin{equation}\label{eq:periodic-pv}
F_M(s)=\lim_{A\to\infty}\int_{[-A,A]}\frac{\mathrm dM(x)}{s-x}
\end{equation}
exists and is $B$-periodic. If $M$ is random and $S$ is uniform on a period and independent of $M$, then, conditionally on $M$,
\[
\frac{F_M(S)}{\pi K}\sim\Cauchy(0,1).
\]
\end{theorem}

\begin{proof}
Represent the finitely many atoms in $[0,B)$ by $B\theta_j/\pi$ and write
\[
M=KB\sum_{j=1}^n w_j\sum_{m\in\Z}\delta_{B(m+\theta_j/\pi)},
\qquad \sum_{j=1}^n w_j=1,\quad \theta_j\in[0,\pi).
\]
The classical partial-fraction expansion
\[
\lim_{N\to\infty}\sum_{m=-N}^N\frac1{z-m}=\pi\cot(\pi z)
\]
converges locally uniformly away from the integers. Replacing the symmetric index cutoff by the condition $|B(m+\theta_j/\pi)|\le A$ changes only a bounded number of endpoint terms, each tending to zero. Consequently
\begin{equation}\label{eq:periodic-cotangent}
\frac{F_M(s)}{\pi K}=\sum_{j=1}^n w_j\cot\left(\frac{\pi s}{B}-\theta_j\right).
\end{equation}
This proves existence and periodicity. Proposition~\ref{prop:cotangent}, conditionally on $M$, proves the distributional statement.
\end{proof}

\begin{corollary}[Polynomials with roots on the unit circle]\label{cor:circle}
Let $P_n$ be a possibly random polynomial of degree $n\ge1$ whose roots lie on the unit circle. If $\Theta$ is uniform on $[0,2\pi)$ and independent of $P_n$, then, conditionally on $P_n$,
\begin{equation}\label{eq:circle-logderivative}
\frac{2\ii}{n}e^{\ii\Theta}\frac{P'_n(e^{\ii\Theta})}{P_n(e^{\ii\Theta})}-\ii
\sim\Cauchy(0,1).
\end{equation}
\end{corollary}

\begin{proof}
Writing the roots, with multiplicity, as $e^{\ii\theta_j}$ gives
\[
\frac{2\ii}{n}e^{\ii\Theta}\frac{P'_n(e^{\ii\Theta})}{P_n(e^{\ii\Theta})}-\ii
=\frac1n\sum_{j=1}^n\cot\left(\frac{\Theta-\theta_j}{2}\right).
\]
Apply Proposition~\ref{prop:cotangent} to the angles $\theta_j/2$ modulo $\pi$.
\end{proof}

\begin{remark}\label{rem:finite-invariance}
If the weighted configuration in Proposition~\ref{prop:cotangent} is invariant in law under common translation of its angles modulo $\pi$, the unconditional Cauchy law still holds for a deterministic observation angle. The same unconditional conclusion holds in Theorem~\ref{thm:periodic} for a stationary periodic measure, and in Corollary~\ref{cor:circle} for a rotationally invariant root configuration. Examples include independent uniform roots and the circular $\beta$-ensemble, for every $\beta>0$.
\end{remark}

\section{Stationary measures and random periodization}\label{sec:stationary}

Retain the signed counting function $N_M$ of the Introduction; thus $\mathrm dN_M=M$, including at zero. Write $\theta_tM(B)=M(B+t)$. A random measure is stationary if $\theta_tM$ has the law of $M$ for every $t\in\R$. We use the canonical space of point measures with its vague Borel sigma-algebra. Let $\I$ be its sigma-field of events invariant under every real shift, and let $\mathcal J$ denote the corresponding field for the unit shift. We work with a random stationary positive point measure with finite intensity
\[
K=\E M((0,1])\in(0,\infty).
\]
Its mean measure is $K\,\mathrm dx$. Hence every deterministic point is almost surely not an atom and, by the locally finite support convention, lies outside the support. Values of $F_M$ at atoms may be assigned arbitrarily whenever spatial integrals are taken.

A measurable real function $f$ is called \emph{shift amenable} if its empirical measures
\[
\frac1L\int_{-L/2}^{L/2}\delta_{f(s)}\,\mathrm ds
\]
converge weakly to a probability measure as $L\to\infty$.

\begin{theorem}[An integrable first-moment criterion]\label{thm:stationary}
Let $M$ be a stationary positive weighted point measure of intensity $K\in(0,\infty)$. Suppose there is a finite nonnegative measurable function $h$ on $[1,\infty)$ such that
\begin{equation}\label{eq:general-discrepancy}
\E|N_M(x)-Kx|\le h(|x|)\quad (|x|\ge1),\qquad
\int_1^\infty\frac{h(r)}{r^2}\,\mathrm dr<\infty.
\end{equation}
Then the following conclusions hold.
\begin{enumerate}[label=(\roman*)]
\item Almost surely, the symmetric principal values
\begin{equation}\label{eq:stationary-pv}
F_M(s)=\lim_{A\to\infty}\int_{[-A,A]}\frac{\mathrm dM(x)}{s-x}
\end{equation}
exist simultaneously for every $s\notin\supp M$. The same limit defines a holomorphic function off the real axis, with locally uniform convergence there, and $G_M=-F_M$ is Herglotz--Pick on the upper half-plane.
\item On one event of probability one, for every $t\in\R$,
\begin{equation}\label{eq:shift-covariance}
F_{\theta_tM}(z)=F_M(z+t)
\end{equation}
off the real axis and at every real point where either side is defined. In particular, $F_M$ is a stationary random function.
\item Almost surely,
\begin{equation}\label{eq:spatial-cauchy}
\frac1L\int_{-L/2}^{L/2}\delta_{F_M(s)/(\pi K)}\,\mathrm ds
\ \Longrightarrow\ \Cauchy(0,1).
\end{equation}
Thus almost every realization is shift amenable, with a nonrandom spatial law.
\item For every deterministic $s$,
\begin{equation}\label{eq:conditional-cauchy}
\Law\left(\left.\frac{F_M(s)}{\pi K}\,\right|\I\right)=\Cauchy(0,1)
\quad\hbox{almost surely}.
\end{equation}
The normalized transform is therefore independent of the translation-invariant information carried by $M$.
\end{enumerate}
\end{theorem}

 The connection between the integrable condition \eqref{eq:general-discrepancy} and a polynomial discrepancy bound will be useful for quantitative estimates. For the construction and periodization, the following pathwise conditions are sufficient:
\begin{equation}\label{eq:pathwise-discrepancy}
D(x):=N_M(x)-Kx=o(|x|),\qquad
\int_\R\frac{|D(x)|}{1+x^2}\,\mathrm dx<\infty.
\end{equation}

\begin{lemma}\label{lem:integrable-discrepancy}
Under \eqref{eq:general-discrepancy}, the conditions \eqref{eq:pathwise-discrepancy} hold almost surely. Moreover, $g(r):=\E|N_M(r)-Kr|$ satisfies $g(r)=o(r)$ as $r\to\infty$.
\end{lemma}

\begin{proof}
For $|x|\le1$, stationarity gives $\E|N_M(x)-Kx|\le2K|x|$. Tonelli's theorem and \eqref{eq:general-discrepancy} therefore prove integrability of $|D(x)|/(1+x^2)$ almost surely.

For positive integers $n$, Birkhoff's theorem applied to the unit-interval masses gives
\[
\frac{M((0,n])}{n}\longrightarrow \E[M((0,1])\mid\mathcal J]
\quad\hbox{almost surely}.
\]
Monotonicity extends this limit to arbitrary real $x\to\infty$. If it differed from $K$, the integral of $|D(x)|/(1+x^2)$ would diverge logarithmically. It therefore equals $K$ almost surely. Apply the same argument to the inverse translation for the negative half-line. Finally, $N_M(r)/r\ge0$ has mean $K$, so
\[
\frac{g(r)}r=2\E\left(K-\frac{N_M(r)}r\right)_+\longrightarrow0
\]
by dominated convergence. Stationarity also gives $\E|N_M(-r)+Kr|=g(r)$.
\end{proof}

\begin{lemma}[Construction and imaginary infinity]\label{lem:construction}
Let $M$ be a positive weighted point measure satisfying \eqref{eq:pathwise-discrepancy} for a constant $K>0$. Then $\int_\R(1+x^2)^{-1}\,\mathrm dM(x)<\infty$, and the principal values in \eqref{eq:stationary-pv} exist simultaneously off the support, with locally uniform convergence off the real axis. Moreover,
\begin{equation}\label{eq:imaginary-infinity}
G_M(\ii\eta)\longrightarrow \ii\pi K\qquad(\eta\to\infty).
\end{equation}
\end{lemma}

\begin{proof}
The relation $N_M(x)=Kx+o(|x|)$ implies $M([-r,r])=O(1+r)$; summing over dyadic annuli gives the asserted weighted mass bound. It also implies $M(\{x\})=o(|x|)$ as $|x|\to\infty$, by bounding an atom by the mass of a neighboring interval of length two.

For cutoffs $A$ which are not atoms, Stieltjes integration by parts gives, off the real axis,
\begin{align*}
\int_{[-A,A]}\frac{\mathrm dM(x)-K\,\mathrm dx}{z-x}
&=\frac{D(A)}{z-A}-\frac{D(-A)}{z+A}
-\int_{-A}^A\frac{D(x)}{(z-x)^2}\,\mathrm dx.
\end{align*}
The boundary terms tend to zero. On a compact subset of $\C\setminus\R$, the integrand is dominated by a constant times $|D(x)|/(1+x^2)$. The reference transform tends to $-\ii\pi K$ on the upper half-plane, and to $\ii\pi K$ on the lower half-plane. This proves locally uniform convergence. The negligible endpoint-atom bound extends the conclusion to arbitrary real cutoffs. Each truncated transform has nonpositive imaginary part on the upper half-plane, so $G_M=-F_M$ is Herglotz--Pick.

For a real $s\notin\supp M$, the difference of the kernels $(s-x)^{-1}$ and $(\ii-x)^{-1}$ is $O_s(x^{-2})$ at infinity, and is integrable on every compact set because $s$ has a neighborhood without atoms. It is therefore absolutely integrable against $M$. Subtracting the transform at $\ii$ proves existence at $s$. For $|s|\le B$ and $|x|\ge2B+2$, the tail bound on this kernel difference is uniform in $s$, so the construction applies simultaneously at all nonatoms.

Finally, the same integration by parts gives
\[
G_M(\ii\eta)=\ii\pi K+\int_\R\frac{D(x)}{(x-\ii\eta)^2}\,\mathrm dx.
\]
For $\eta\ge1$, the absolute value of the integrand is bounded by $|D(x)|/(1+x^2)$ and tends pointwise to zero. Dominated convergence proves \eqref{eq:imaginary-infinity}.
\end{proof}

\begin{lemma}[Real boundary values]\label{lem:boundary}
Under \eqref{eq:pathwise-discrepancy}, every $s\notin\supp M$ satisfies
\begin{equation}\label{eq:boundary-square}
\int_\R\frac{\mathrm dM(x)}{(s-x)^2}<\infty,
\qquad
|F_M(s+\ii\eta)-F_M(s)|
\le\eta\int_\R\frac{\mathrm dM(x)}{(s-x)^2}\quad(\eta>0).
\end{equation}
In particular, the principal value is the vertical boundary value. The transform extends meromorphically across the real axis, with poles only at the atoms, and
\begin{equation}\label{eq:derivative-transform}
G'_M(s)=\int_\R\frac{\mathrm dM(x)}{(s-x)^2}.
\end{equation}
\end{lemma}

\begin{proof}
There is no mass near $s$, and $(s-x)^{-2}$ is bounded on the remainder of a fixed compact set. At infinity it is bounded by a constant times $(1+x^2)^{-1}$. This proves the first assertion. The identity
\[
\frac1{s+\ii\eta-x}-\frac1{s-x}
=\frac{-\ii\eta}{(s+\ii\eta-x)(s-x)}
\]
has absolute value at most $\eta/(s-x)^2$. Integrate and pass to symmetric limits to obtain the inequality. On a small complex neighborhood of a nonatom, the corresponding kernel differences and their derivatives are uniformly dominated by an integrable function. They thus give analytic continuation and differentiation under the integral. Near an atom, subtract its single pole first and apply the same argument to the remaining locally finite support. This proves the meromorphic assertion and \eqref{eq:derivative-transform}.
\end{proof}

\begin{lemma}[Pathwise shift covariance]\label{lem:covariance}
If \eqref{eq:pathwise-discrepancy} holds, it holds for every translate $\theta_tM$, with the same $K$, and \eqref{eq:shift-covariance} holds for all $t\in\R$ simultaneously.
\end{lemma}

\begin{proof}
The discrepancy based at the translated origin is
\begin{equation}\label{eq:shifted-discrepancy}
D_t(x)=N_{\theta_tM}(x)-Kx=D(x+t)-D(t).
\end{equation}
It is $o(|x|)$, and its weighted absolute integral is finite by a change of variables and comparison of $1+x^2$ with $1+(x+t)^2$. Lemma~\ref{lem:construction} therefore constructs the translated transform.

After changing variables, its truncation is the original transform at $z+t$, integrated over $[-A+t,A+t]$. For each fixed $c>0$, the mass of any interval of length at most $c$ adjacent to either endpoint $\pm A$ is $o(A)$: this follows from $N_M(x)=Kx+o(|x|)$, with the endpoint-atom bound used in Lemma~\ref{lem:construction}. The difference between the two truncations is consequently $o(1)$, because their kernels at these endpoints are $O_z(A^{-1})$. This argument is deterministic and applies to every fixed real $t$ on the same realization, proving the simultaneous assertion.
\end{proof}

\begin{proposition}[A deterministic periodization principle]\label{prop:periodization}
Let $M$ satisfy \eqref{eq:pathwise-discrepancy}. Let $M_R$ be the $R$-periodic extension of its restriction to $(-R/2,R/2]$, and set
\[
K_R=\frac{M((-R/2,R/2])}{R}.
\]
Then $K_R\to K$. For each $\varepsilon\in(0,1/2)$ and $B<\infty$,
\begin{equation}\label{eq:interior-periodization}
\sup_{|s|\le(1/2-\varepsilon)R}\ 
\sup_{|z|\le B}|F_M(s+z)-F_{M_R}(s+z)|\longrightarrow0.
\end{equation}
At common poles the difference in this formula is understood after cancelling the corresponding infinite terms. Moreover,
\begin{equation}\label{eq:deterministic-spatial-law}
\mu_R^M:=\frac1R\int_{-R/2}^{R/2}\delta_{F_M(s)/(\pi K)}\,\mathrm ds
\ \Longrightarrow\ \Cauchy(0,1).
\end{equation}
\end{proposition}

\begin{proof}
Put $c_R=K-K_R$ and $D_R(x)=N_{M_R}(x)-K_Rx$. The density hypothesis gives $c_R=o(1)$. The function $D_R$ is $R$-periodic, and on the central period it equals $D(x)+c_Rx$. Since
\[
\sup_{|x|\le R/2}|D(x)|=o(R),
\]
we have $e_R:=\sup_\R|D_R|=o(R)$. The signed counting function of $M-M_R$ is
\[
H_R(x)=N_M(x)-N_{M_R}(x)=c_Rx+D(x)-D_R(x),
\]
and vanishes on the interior of the central period.

For $|u|<R/2$, integrate by parts at a symmetric outer cutoff $A$ and pair the positive and negative tails. The linear boundary terms sum to $2c_RAu/(u^2-A^2)$ and cancel in the limit; the remaining terms tend to zero by $D=o(|x|)$ and boundedness of $D_R$. Thus
\begin{equation}\label{eq:paired-periodization}
F_M(u)-F_{M_R}(u)
=-\int_{R/2}^\infty\left\{
\frac{H_R(x)}{(u-x)^2}+\frac{H_R(-x)}{(u+x)^2}
\right\}\,\mathrm dx.
\end{equation}
The linear terms in this integral are paired before integration. The resulting integral is absolutely convergent; it also defines the analytic continuation of the difference through its cancelled poles in $|u|<R/2$.

If $|u|\le(1/2-\varepsilon)R$, then $x-|u|\ge2\varepsilon x$ for $x\ge R/2$. The terms involving $D$ in \eqref{eq:paired-periodization} are bounded by
\[
\frac1{4\varepsilon^2}\int_{|x|\ge R/2}\frac{|D(x)|}{x^2}\,\mathrm dx=o(1),
\]
uniformly in $u$. The terms involving $D_R$ are bounded by $2e_R/(R/2-|u|)=o(1)$. For the linear terms, use
\[
x\left|\frac1{(u-x)^2}-\frac1{(u+x)^2}\right|
\le\frac{4|u|}{(x-|u|)^2}.
\]
Indeed, one of $|x+u|$ and $|x-u|$ is at least $x$, and the other is at least $x-|u|$. Their contribution is at most
\[
\frac{4|u||c_R|}{R/2-|u|}=O_\varepsilon(|c_R|)=o(1).
\]
Replacing $u$ by $s+z$ and, for sufficiently large $R$, replacing $\varepsilon$ by $\varepsilon/2$, proves \eqref{eq:interior-periodization}.

Let $S_R$ be uniform on $(-R/2,R/2]$. For fixed $\varepsilon$, the complement of $|S_R|\le(1/2-\varepsilon)R$ has probability $2\varepsilon$, whereas the difference of transforms converges uniformly to zero on this event. First let $R\to\infty$ and then $\varepsilon\downarrow0$. We obtain
\[
F_M(S_R)-F_{M_R}(S_R)\longrightarrow0
\quad\hbox{in probability with respect to }S_R.
\]
For all sufficiently large $R$, $K_R>0$. Theorem~\ref{thm:periodic} gives $F_{M_R}(S_R)/(\pi K_R)\sim\Cauchy(0,1)$ exactly. Since $K_R\to K$, Slutsky's theorem proves \eqref{eq:deterministic-spatial-law}.
\end{proof}

\begin{proof}[Proof of Theorem~\ref{thm:stationary}]
Lemma~\ref{lem:integrable-discrepancy} supplies a probability-one event on which the deterministic hypotheses \eqref{eq:pathwise-discrepancy} hold. Lemmas~\ref{lem:construction} and \ref{lem:covariance} prove the construction and covariance assertions there. Proposition~\ref{prop:periodization} gives the spatial law.

For measurability, integrate uniformly bounded continuous approximations, supported on $[-n-1,n+1]$, to $-x^{-1}\mathbf1_{\{1/m<|x|<n\}}$, then let $m\to\infty$ followed by $n\to\infty$. The deterministic endpoints are almost surely atom-free and the support avoids zero, so these iterated limits give a measurable version of $F_M(0)$; set it to zero if they fail. The measurable translation action gives joint measurability in $M,t$.

For the conditional assertion, let $Y$ be bounded and $\I$-measurable and let $\varphi$ be bounded and continuous. For each deterministic $t$, covariance gives $F_M(t)=F_{\theta_tM}(0)$, while $Y\circ\theta_t=Y$ almost surely. Stationarity therefore yields
\[
\E\left[Y\varphi\left(\frac{F_M(t)}{\pi K}\right)\right]
=\E\left[Y\varphi\left(\frac{F_M(0)}{\pi K}\right)\right].
\]
Average over $t\in[-L/2,L/2]$ and apply \eqref{eq:spatial-cauchy} and dominated convergence. The result is
\[
\E\left[Y\varphi\left(\frac{F_M(0)}{\pi K}\right)\right]
=\E[Y]\int_\R\frac{\varphi(x)}{\pi(1+x^2)}\,\mathrm dx.
\]
A countable convergence-determining family of test functions identifies the conditional probability law. Stationarity gives the same assertion at every deterministic $s$.
\end{proof}

The unit-shift field $\mathcal J$ cannot replace $\I$ in \eqref{eq:conditional-cauchy}. For $M=\sum_{n\in\Z}\delta_{n+U}$ with $U$ uniform on $[0,1)$, the unit shift fixes $M$, so $\mathcal J=\sigma(U)$ modulo null sets, whereas $\I$ is trivial. Thus $F_M(0)=-\pi\cot(\pi U)$ is $\mathcal J$-measurable.

In the case where all the weights 
involved in the measure $M$ are integers, we can interpret the Stieltjes transform on $M$ in terms of the logarithmic derivative of a random holomorphic function. The following result holds: 
\begin{proposition}\label{prop:coordinate}
Under the assumptions of  Theorem~\ref{thm:stationary}, fix a deterministic $s$ and put
\[
\Xi_s=\theta_sM,\qquad a_s=F_M(s),\qquad
\delta_s=\int_\R x^{-2}\,\mathrm d\Xi_s(x).
\]
Almost surely, $\delta_s<\infty$ and
\begin{equation}\label{eq:centered-herglotz}
G_M(s+z) = - F_M (s+z) = -a_s+\int_\R\left(\frac1{x-z}-\frac1x\right)\mathrm d\Xi_s(x),
\qquad z\in\C\setminus\R.
\end{equation}
In particular, the linear Herglotz coefficient is zero, and $a_s/(\pi K)$ is standard Cauchy conditionally on $\I$.

If $M$ is integer-valued and $E_1(w)=(1-w)e^w$, the product
\begin{equation}\label{eq:canonical-product}
\Psi_s(z)=e^{a_sz}\prod_{x\in\Xi_s}E_1(z/x)
\end{equation}
converges locally uniformly, and defines an entire function with $\Psi_s(0)=1$, and $-\Psi'_s/\Psi_s=G_M(s+\cdot)$. Here, $\prod_{x\in\Xi_s}$ denotes the product on the support of $\Xi_s$, the factor indexed by $x$ being repeated according to the multiplicity $\Xi_s(\{x\})$. 

\end{proposition}

\begin{proof}
The point $s$ is almost surely outside the support of $M$, and finiteness of $\delta_s$ follows from Lemma~\ref{lem:boundary}. By shift covariance, the symmetric principal value of $\int x^{-1}\,\mathrm d\Xi_s(x)$ is $-a_s$. The kernel difference in \eqref{eq:centered-herglotz} is absolutely integrable: there is no mass near zero, and it is $O_z(x^{-2})$ at infinity. Subtracting the value at zero from the shifted transform therefore gives the formula. The conditional law is Theorem~\ref{thm:stationary}.

For integer masses, $\sum_{x\in\Xi_s}x^{-2}<\infty$, with multiplicities. Since $\log E_1(w)=O(w^2)$ for $|w|\le1/2$, the product converges uniformly on compact sets and defines an entire function. Its logarithmic derivative, away from its zeros, is
\[
-\frac{\Psi'_s(z)}{\Psi_s(z)}
=-a_s+\sum_{x\in\Xi_s}\left(\frac1{x-z}-\frac1x\right),
\]
where the sum converges locally uniformly away from the support. This is \eqref{eq:centered-herglotz}.
\end{proof}

To specify the real coordinate topology used below, let $\mathscr X$ consist of triples $(\Xi,a,\delta)$ with $\Xi$ an integer-valued positive point measure, $0\notin\supp\Xi$,
\[
S(\Xi):=\int_\R x^{-2}\,\mathrm d\Xi(x)<\infty,\qquad a\in\R,
\qquad \delta\ge S(\Xi).
\]
The empty measure is allowed. We give $\mathscr X$ the subspace topology of vague convergence of $\Xi$ and ordinary convergence of $a$ and $\delta$. The quantity $q=\delta-S(\Xi)$ is the quadratic defect. The associated formulas are
\begin{align}\label{eq:completed-formulas}
W(z)&=-a+qz+\int_\R\left(\frac1{x-z}-\frac1x\right)\mathrm d\Xi(x),\\
\Psi(z)&=e^{az-qz^2/2}\prod_{x\in\Xi}E_1(z/x).\notag
\end{align}
These are completed real spectral coordinates of the type considered in \cite{Companion}. All states constructed from the principal values in this section have $q=0$. The proof below uses the stated topology and the formulas directly. We write $\Hol(\Omega)$ for holomorphic functions on $\Omega$, with uniform convergence on compact subsets.

\begin{theorem}[Field-valued random periodization]\label{thm:field-periodization}
Under the assumptions of Theorem~\ref{thm:stationary}, let $M_R$ be the periodization in Proposition~\ref{prop:periodization}, and let $S_R$ be independent of $M$ and uniform on $(-R/2,R/2]$. For almost every realization of $M$ and every compact $C\subset\C\setminus\R$,
\begin{equation}\label{eq:field-coupling}
\sup_{z\in C}|G_{M_R}(S_R+z)-G_M(S_R+z)|\longrightarrow0
\end{equation}
in probability with respect to $S_R$. Consequently, under the joint law of $M$ and $S_R$,
\begin{equation}\label{eq:field-convergence}
G_{M_R}(S_R+\cdot)\ \Longrightarrow\ G_M
\quad\hbox{in }\Hol(\C\setminus\R).
\end{equation}

If $M$ is integer-valued, set
\[
\Xi_R=\theta_{S_R}M_R,\qquad a_R=F_{M_R}(S_R),\qquad
\delta_R=\int_\R\frac{\mathrm dM_R(x)}{(x-S_R)^2},
\]
and $\Xi=M$, $a=F_M(0)$, $\delta=\int_\R x^{-2}\,\mathrm dM(x)$. Then
\begin{equation}\label{eq:state-convergence}
(\Xi_R,a_R,\delta_R)\ \Longrightarrow\ (\Xi,a,\delta)
\quad\hbox{in }\mathscr X,
\end{equation}
and
\begin{equation}\label{eq:product-convergence}
e^{a_Rz}\prod_{x\in\Xi_R}E_1(z/x)
\ \Longrightarrow\ e^{az}\prod_{x\in\Xi}E_1(z/x)
\quad\hbox{in }\Hol(\C).
\end{equation}
Every state in \eqref{eq:state-convergence} has zero quadratic defect.
\end{theorem}

\begin{proof}
Work first with a fixed realization satisfying \eqref{eq:pathwise-discrepancy}. For fixed $\varepsilon>0$, Proposition~\ref{prop:periodization} gives uniform convergence on each fixed $z$-disk whenever $|S_R|\le(1/2-\varepsilon)R$. The complement has probability $2\varepsilon$. Letting $R\to\infty$ and then $\varepsilon\downarrow0$ proves \eqref{eq:field-coupling}. A countable exhaustion by compact sets gives convergence in probability for a metric defining the compact-open topology. Averaging the conditional probabilities proves the same coupling convergence under the joint law. By stationarity and covariance, $G_M(S_R+\cdot)$ has exactly the law of $G_M$. This proves \eqref{eq:field-convergence}.

For the remaining statements, compare the periodic state to
\[
\widetilde\Xi_R=\theta_{S_R}M,\qquad
\widetilde a_R=F_M(S_R),\qquad
\widetilde\delta_R=\int_\R\frac{\mathrm dM(x)}{(x-S_R)^2}.
\]
This triple has exactly the law of $(\Xi,a,\delta)$. Conditional uniformity ensures that $S_R$ is not an atom; stationarity does the same at zero. The required square-integrability of the kernels follows from Lemma~\ref{lem:boundary}, and, for $M_R$, from periodicity.

On $|S_R|\le(1/2-\varepsilon)R$, the two translated point measures agree on $(-\varepsilon R,\varepsilon R)$. Thus
\[
V_R(z):=G_{M_R}(S_R+z)-G_M(S_R+z)
\]
has all its poles cancelled in that disk and is holomorphic there. The proof of \eqref{eq:interior-periodization} applies across the cancelled poles and gives $V_R\to0$ uniformly on every fixed disk. In particular,
\[
V_R(0)=-a_R+\widetilde a_R\longrightarrow0.
\]
Cauchy's integral formula on a fixed circle about zero also gives $V'_R(0)\to0$. By \eqref{eq:derivative-transform},
\[
V'_R(0)=\delta_R-\widetilde\delta_R.
\]
As before, the exceptional proportion $2\varepsilon$ can be made arbitrarily small. We conclude that these two coordinate differences tend to zero in joint probability. The point measures agree on every fixed compact with probability tending to one, which gives convergence to zero of their vague distance. The same coupling argument as for the fields now proves \eqref{eq:state-convergence} in the explicitly specified coordinate topology. The definition of $\delta_R$ and $\delta$ proves that their defects vanish.

Finally, write $\Psi_R$ for the product on the left of \eqref{eq:product-convergence}, and write $\widetilde\Psi_R=\Psi_{S_R}$ for the product associated with the shifted original measure. On a fixed disk, on the interior event just used, these products have identical zeros with identical multiplicities. Their quotient therefore extends to a holomorphic nonvanishing function on the disk. It has value one at zero and logarithmic derivative $-V_R$. Consequently
\[
\frac{\Psi_R(z)}{\widetilde\Psi_R(z)}
=\exp\left(-\int_0^z V_R(w)\,\mathrm dw\right).
\]
This quotient converges uniformly to one on every fixed disk in joint probability. Stationarity gives $\widetilde\Psi_R\stackrel{d}=\Psi_0$ as random entire functions, so their suprema on a fixed disk form a tight family. Multiplication by these suprema shows $\Psi_R-\widetilde\Psi_R\to0$ uniformly there in probability. A countable exhaustion proves convergence in the compact-open topology and hence \eqref{eq:product-convergence}.
\end{proof}

\begin{remark}[Zero defect and imaginary infinity]\label{rem:zero-defect}
For a field $W$ in \eqref{eq:completed-formulas},
\[
\frac{\operatorname{Im}W(\ii\eta)}\eta
=q+\int_\R\frac{\mathrm d\Xi(x)}{x^2+\eta^2}\longrightarrow q.
\]
Dominated convergence applies because $\int x^{-2}\,\mathrm d\Xi<\infty$. Thus a finite limit at imaginary infinity forces $q=0$. Zero defect alone does not give a nondegenerate Cauchy boundary law; shift amenability and the limiting value must be checked separately. For example, a finite point measure has zero defect, but its transform tends to zero under spatial shifts and its spatial law is a point mass. Also, $q=0$ means $\delta=S(\Xi)$, not $\delta=0$.
\end{remark}

The deterministic density in Theorem~\ref{thm:stationary} excludes variation of the intensity between invariant components. When this variation is allowed, the same pathwise argument identifies the resulting scale mixture.

\begin{theorem}[Componentwise Cauchy law]\label{thm:components}
Let $M$ be a stationary positive weighted point measure with $\E M((0,1])<\infty$, and put
\[
K_{\I}=\E[M((0,1])\mid\I].
\]
Assume $K_{\I}>0$ almost surely and, with $D_{\I}(x)=N_M(x)-K_{\I}x$, assume almost surely that
\begin{equation}\label{eq:component-discrepancy}
D_{\I}(x)=o(|x|),\qquad
\int_\R\frac{|D_{\I}(x)|}{1+x^2}\,\mathrm dx<\infty.
\end{equation}
Then the principal values exist simultaneously at every nonatom, are pathwise shift covariant, and, for every deterministic $s$,
\begin{equation}\label{eq:component-cauchy}
\Law\left(\left.\frac{F_M(s)}{\pi K_{\I}}\,\right|\I\right)=\Cauchy(0,1)
\quad\hbox{almost surely}.
\end{equation}
In particular, $F_M(s)=\pi K_{\I}X$ with $X$ standard Cauchy and independent of $\I$. If $K=\E M((0,1])$, then $F_M(s)/(\pi K)$ is standard Cauchy if and only if $K_{\I}=K$ almost surely.

A sufficient condition for \eqref{eq:component-discrepancy} is an almost-sure bound
\begin{equation}\label{eq:component-power}
|N_M(x)-K_{\I}x|\le C_M(1+|x|^\beta),\qquad x\in\R,
\end{equation}
where $\beta\in[0,1)$ is deterministic and $C_M$ is almost surely finite.
\end{theorem}

\begin{proof}
For a realization satisfying \eqref{eq:component-discrepancy}, apply Lemmas~\ref{lem:construction} and \ref{lem:covariance} with the constant $K$ replaced by $\kappa=K_{\I}>0$. Proposition~\ref{prop:periodization} gives
\[
\frac1R\int_{-R/2}^{R/2}\delta_{F_M(u)/(\pi K_{\I})}\,\mathrm du
\ \Longrightarrow\ \Cauchy(0,1)
\quad\hbox{almost surely}.
\]
The asymptotic density provides a version of $K_{\I}$ invariant under all translations on this event. For bounded $\I$-measurable $Y$ and bounded continuous $\varphi$, stationarity therefore gives
\[
\E\left[Y\varphi\left(\frac{F_M(u)}{\pi K_{\I}}\right)\right]
=\E\left[Y\varphi\left(\frac{F_M(0)}{\pi K_{\I}}\right)\right].
\]
Averaging in $u$ and using dominated convergence proves the conditional law as in Theorem~\ref{thm:stationary}. In particular, the variable $X=F_M(s)/(\pi K_{\I})$ is independent of $\I$.

The characteristic function after normalization by the mean intensity is
\[
\E\exp\left(\ii t\frac{F_M(s)}{\pi K}\right)
=\E\exp\left(-|t|\frac{K_{\I}}K\right).
\]
Since $\E[K_{\I}/K]=1$, strict convexity of $v\mapsto e^{-|t|v}$ for any $t\ne0$ shows that this equals $e^{-|t|}$ only when $K_{\I}/K$ is constant almost surely. The converse is immediate. Finally, \eqref{eq:component-power} implies both assertions in \eqref{eq:component-discrepancy}, since $\beta<1$.
\end{proof}

\begin{example}[A visible component and an independent transform]\label{ex:equal-intensity-mixture}
Let $B$ be Bernoulli with parameter $1/2$. Conditional on $B=0$, let $M$ be a Poisson process of intensity one; conditional on $B=1$, let it be the sine-kernel determinantal process with kernel
\[
K_{\mathrm{sine}}(x,y)=\frac{\sin\pi(x-y)}{\pi(x-y)}.
\]
Both laws are stationary and ergodic, and both have intensity one. The sine process is mixing \cite{Soshnikov}. The laws are distinct: for instance, the number variance on a fixed nonempty interval is strictly smaller than its length for the sine process, and equals its length for the Poisson process. Distinct ergodic stationary laws are mutually singular on invariant events; equivalently, spatial averages of a bounded observable with different expectations distinguish them. Thus $B$ is measurable, modulo null sets, with respect to the invariant sigma-field of the mixture; see also \cite[Chapter~10]{Kallenberg}.

The number variances over an interval of length $L$ are $L$ and $O(\log(2+L))$, respectively. Cauchy--Schwarz therefore gives $\E|M((0,L])-L|=O(1+L^{1/2})$ under the mixture. Theorem~\ref{thm:stationary} applies, and
\[
\Law\left(\left.\frac{F_M(s)}\pi\,\right|B\right)=\Cauchy(0,1).
\]
The infinite configuration identifies the component, whereas its one-point transform is independent of that information. Nonergodicity alone does not create a scale mixture.
\end{example}

\begin{example}[A nontrivial scale mixture]\label{ex:random-intensity}
Let $\Lambda>0$ have finite mean, and conditionally on $\Lambda$ let $M$ be a Poisson process of intensity $\Lambda$. The conditional process is ergodic and its asymptotic density is $\Lambda$, so $K_{\I}=\Lambda$. To check \eqref{eq:component-power}, fix $\beta\in(1/2,1)$ and condition on $\Lambda=\lambda$. The Poisson exponential bound gives, for all sufficiently large integers $n$,
\[
\P\bigl(|M((0,n])-\lambda n|>n^\beta\mid\Lambda=\lambda\bigr)
\le2\exp\left(-\frac{n^{2\beta}}{2(\lambda n+n^\beta)}\right).
\]
For $Z\sim\operatorname{Poisson}(\mu)$, $\log\E e^{t(Z-\mu)}=\mu(e^t-1-t)$. Chernoff's bound and $(1+v)\log(1+v)-v\ge v^2/[2(1+v)]$ for $v\ge0$ give the upper tail; $\log\E e^{-t(Z-\mu)}\le\mu t^2/2$ for $t\ge0$ gives the lower tail.
These probabilities are summable. Borel--Cantelli on both half-lines, followed by monotonic interpolation between integer arguments, proves \eqref{eq:component-power}, with an almost-surely finite random constant. Theorem~\ref{thm:components} yields
\[
F_M(s)=\pi\Lambda X,
\]
where $X$ is standard Cauchy and independent of $\Lambda$. Normalization by the mean intensity gives a standard Cauchy law precisely when $\Lambda$ is constant. If $\Lambda$ has finite second moment, the conditional variance formula gives
\[
\operatorname{Var}N_M(x)=\E[\Lambda]|x|+\operatorname{Var}(\Lambda)x^2.
\]
If its second moment is infinite, the variance on every nonempty bounded interval is infinite. 
\end{example}

There following corollaries give simple conditions under which the assumptions, and then the conclusions, of Theorem~\ref{thm:stationary} are satisfied. 
\begin{corollary}[Polynomial first moments]\label{cor:polynomial}
All conclusions of Theorem~\ref{thm:stationary} hold if, for some $C<\infty$ and $\alpha\in[0,1)$,
\begin{equation}\label{eq:polynomial-first-moment}
\E|N_M(x)-Kx|\le C(1+|x|^\alpha),\qquad x\in\R.
\end{equation}
\end{corollary}

\begin{proof}
The function $h(r)=C(1+r^\alpha)$ satisfies \eqref{eq:general-discrepancy}.
\end{proof}

\begin{corollary}[Subquadratic variance]\label{cor:variance}
Let $M$ be stationary with finite positive intensity. If, for some $\alpha\in[0,1)$,
\begin{equation}\label{eq:variance-condition}
\operatorname{Var}N_M(x)\le C(1+|x|^{2\alpha}),\qquad x\in\R,
\end{equation}
then \eqref{eq:polynomial-first-moment} holds with the same $\alpha$, and Theorem~\ref{thm:stationary} applies.
\end{corollary}

\begin{proof}
Stationarity gives $\E N_M(x)=Kx$. Cauchy--Schwarz yields \eqref{eq:polynomial-first-moment}, after changing its constant.
\end{proof}

\begin{lemma}[An almost-sure polynomial bound]\label{lem:polynomial-discrepancy}
Under \eqref{eq:polynomial-first-moment}, for every $\beta\in((1+\alpha)/2,1)$ there is an almost-surely finite $C_M$ such that
\begin{equation}\label{eq:pathwise-power}
|N_M(x)-Kx|\le C_M(1+|x|^\beta),\qquad x\in\R.
\end{equation}
\end{lemma}

\begin{proof}
Choose $q>0$ with $q(\beta-\alpha)>1$ and $q(1-\beta)\le1$. Such a choice is possible because $\beta>(1+\alpha)/2$. Markov's inequality gives
\[
\sum_{n\ge1}\P\bigl(|N_M(n^q)-Kn^q|>n^{q\beta}\bigr)
\le C\sum_{n\ge1}\bigl(n^{-q\beta}+n^{-q(\beta-\alpha)}\bigr)<\infty.
\]
The same estimate holds at $-n^q$. Borel--Cantelli proves the desired bound at both grids. If $n^q\le x\le(n+1)^q$, monotonicity bounds the discrepancy at $x$ by the two grid discrepancies and $K((n+1)^q-n^q)$. This last term is $O(n^{q-1})=O(n^{q\beta})$ by the choice of $q$. Interpolation on the negative half-line is identical. Finitely many exceptional grid points and the remaining compact interval are absorbed into $C_M$.
\end{proof}

For probability measures on $\R$, let $\dBL$ denote the supremum of the difference of integrals over real test functions bounded in absolute value by one and with Lipschitz constant at most one.

\begin{proposition}[A quantitative spatial law]\label{prop:spatial-rate}
Under \eqref{eq:polynomial-first-moment}, for every $\beta\in((1+\alpha)/2,1)$, almost surely,
\begin{equation}\label{eq:spatial-rate}
\dBL\bigl(\mu_R^M,\Cauchy(0,1)\bigr)
=O_M\bigl(R^{\beta-1}\log(2+R)\bigr).
\end{equation}
Here $\mu_R^M$ is defined in \eqref{eq:deterministic-spatial-law}. On the same event, for $|u|<R/2$,
\begin{equation}\label{eq:sharp-periodization-bound}
|F_M(u)-F_{M_R}(u)|
\le C_M\frac{R^\beta}{R/2-|u|},\qquad R\ge1,
\end{equation}
with cancellation at common poles.
\end{proposition}

\begin{proof}
Work on the event of Lemma~\ref{lem:polynomial-discrepancy}. In the notation of Proposition~\ref{prop:periodization},
\[
|c_R|=O_M(R^{\beta-1}),\qquad \sup|D_R|=O_M(R^\beta).
\]
For $x\ge R/2$, write $H_R(x)=c_Rx+r_R^+(x)$ and $H_R(-x)=-c_Rx+r_R^-(x)$, where
$|r_R^\pm(x)|\le C_M(1+x^\beta)$. Put $d=R/2-|u|$. Pairing the linear terms in \eqref{eq:paired-periodization} bounds their contribution by
$4|u||c_R|/d=O_M(R^\beta/d)$. For the remainder, split the integral at $R$:
\[
\int_{R/2}^R\frac{1+x^\beta}{(x-|u|)^2}\,\mathrm dx
\le C\frac{R^\beta}{d},\qquad
\int_R^\infty\frac{1+x^\beta}{(x-|u|)^2}\,\mathrm dx
\le C_\beta R^{\beta-1}\le C_\beta\frac{R^\beta}{d}.
\]
This proves \eqref{eq:sharp-periodization-bound}.

Couple the original and periodic transforms using a common uniform point. For any test function in the definition of $\dBL$, their normalized integrals differ by at most
\[
\frac1R\int_{-R/2}^{R/2}
\min\left(2,\frac{C_MR^\beta}{\pi K(R/2-|s|)}\right)ds
=O_M(R^{\beta-1}\log(2+R)).
\]
To see the last bound, change variables to the distance from an endpoint and split the integral where that distance is a constant multiple of $R^\beta$. The contribution below the splitting point is $O_M(R^{\beta-1})$, and the remaining integral is logarithmic.

The periodic normalized transform has the Cauchy density of scale $r_R=K_R/K=1+O_M(R^{\beta-1})$. If $f_r(x)=r/[\pi(r^2+x^2)]$, then
\[
\int_\R|\partial_r f_r(x)|\,\mathrm dx=\frac2{\pi r}.
\]
Hence $\int|f_r-f_1|\le(2/\pi)|\log r|$. The difference between the periodic law and the standard Cauchy law therefore costs $O_M(R^{\beta-1})$. Combining the two estimates proves \eqref{eq:spatial-rate}.
\end{proof}

\section{Quantitative and asymmetric truncations}\label{sec:truncation}

For a deterministic observation point $s$, define the centered truncation
\[
F^{\mathrm{sym}}_{M,A}(s)=\int_{[s-A,s+A]}\frac{\mathrm dM(x)}{s-x}.
\]
The observation point and any fixed endpoints are almost surely not atoms
for a stationary measure. 
On the assumptions of  Theorem~\ref{thm:stationary}, we have $\lim_{A\to\infty}F^{\mathrm{sym}}_{M,A}(s)=F_{\theta_sM}(0)=F_M(s)$.

\begin{theorem}[Quantitative symmetric truncation]\label{thm:truncation}
Under the assumptions of Theorem~\ref{thm:stationary}, for every deterministic $s$ and $A\ge1$,
\begin{equation}\label{eq:general-truncation-rate}
\E|F_M(s)-F^{\mathrm{sym}}_{M,A}(s)|
\le T_h(A):=\frac{2h(A)}A+2\int_A^\infty\frac{h(x)}{x^2}\,\mathrm dx.
\end{equation}
Consequently,
\begin{equation}\label{eq:truncation-bl}
\dBL\left(\Law\left(\frac{F^{\mathrm{sym}}_{M,A}(s)}{\pi K}\right),\Cauchy(0,1)\right)
\le\frac{T_h(A)}{\pi K}.
\end{equation}
Both bounds hold with $h$ replaced by $g$ from Lemma~\ref{lem:integrable-discrepancy}, for which $T_g(A)\to0$; no sublinearity of the chosen envelope $h$ is required.
Under \eqref{eq:polynomial-first-moment}, one has $T_h(A)=O_{C,\alpha}(A^{\alpha-1})$. The logarithmic choice $h(x)=C\sqrt{\log(2+x)}$ gives
\[
T_h(A)=O_C\left(\frac{\sqrt{\log(2+A)}}A\right).
\]
\end{theorem}

\begin{proof}
Stationarity reduces the proof to $s=0$. Put
\[
Q(x)=M((0,x])-M([-x,0)),\qquad x>0.
\]
The intensity contributions cancel, and for each deterministic $x\ge1$ the absence of endpoint atoms gives $\E|Q(x)|\le2h(x)$. As a function of $x$, $Q$ is the signed distribution function of the radial positive-side mass minus the radial negative-side mass. Stieltjes integration by parts therefore yields the exact identity
\begin{equation}\label{eq:imbalance-ibp}
\int_{A<|x|\le B}\frac{\mathrm dM(x)}x
=\frac{Q(B)}B-\frac{Q(A)}A+\int_A^B\frac{Q(x)}{x^2}\,\mathrm dx.
\end{equation}
The first term tends to zero almost surely, by \eqref{eq:pathwise-discrepancy} and its endpoint-atom consequence, and in $L^1$, since $\E|Q(B)|\le2g(B)=o(B)$ by Lemma~\ref{lem:integrable-discrepancy}. The integral on the right converges absolutely almost surely and in $L^1$ by Tonelli. Let $B\to\infty$, take absolute values and expectations, and recall the sign $F_M(0)=-\pv\int x^{-1}\,\mathrm dM(x)$. This proves \eqref{eq:general-truncation-rate}, also with $g$ in place of $h$. Since $g(A)/A\to0$ and $g\le h$, we have $T_g(A)\to0$.

Coupling the truncation with its limit and testing against the functions defining $\dBL$ proves \eqref{eq:truncation-bl}. For polynomial $h$, integrate $x^{\alpha-2}$ directly. For the logarithmic bound, substitute $x=At$ and use
\[
\log(2+At)\le\log(2+A)+\log t,\qquad t\ge1.
\]
The inequality $\sqrt{a+b}\le\sqrt a+\sqrt b$ and finiteness of $\int_1^\infty t^{-2}\sqrt{\log t}\,\mathrm dt$ give the stated rate.
\end{proof}

\begin{theorem}[Asymmetric cutoffs]\label{thm:asymmetric}
Under the assumptions of Theorem~\ref{thm:stationary}, almost surely, for every $s\notin\supp M$,
\begin{equation}\label{eq:asymmetric-limit}
\int_{[s-L,s+R]}\frac{\mathrm dM(x)}{s-x}-K\log\frac LR
\longrightarrow F_M(s)\qquad(L,R\to\infty).
\end{equation}
In particular, if $L_n,R_n\to\infty$ and $\log(L_n/R_n)\to\ell\in\R$, then, for every deterministic $s$,
\begin{equation}\label{eq:asymmetric-law}
\frac1{\pi K}\int_{[s-L_n,s+R_n]}\frac{\mathrm dM(x)}{s-x}
\ \Longrightarrow\ \Cauchy(\ell/\pi,1).
\end{equation}
\end{theorem}

\begin{proof}
Work on the probability-one event where \eqref{eq:pathwise-discrepancy} holds and fix $s$ outside the support. Translate by $s$, and denote its discrepancy by $D_s$ as in \eqref{eq:shifted-discrepancy}. It is $o(|x|)$ with integrable weighted absolute value. Fix a symmetric inner cutoff $B$ and compare the asymmetric truncation, for $L,R>B$, to the truncation on $[s-B,s+B]$. The two outer integrals of the reference measure $K\,\mathrm dx$ are
\[
K\log\frac LB+K\log\frac BR=K\log\frac LR.
\]
It remains to consider the corresponding outer integrals of $\mathrm dD_s(y)/(-y)$ over $[-L,-B]$ and $[B,R]$. Integration by parts bounds them in absolute value by
\[
\frac{|D_s(-L)|}{L}+\frac{|D_s(R)|}{R}
+\frac{|D_s(-B)|+|D_s(B)|}{B}
+\int_{B<|y|<\max(L,R)}\frac{|D_s(y)|}{y^2}\,dy,
\]
initially with nonatom endpoints. Endpoint atoms give additional terms tending to zero at large cutoffs and may equivalently be handled by one-sided distribution functions. Take the upper limit as $L,R\to\infty$. Then let $B\to\infty$ through nonatoms. The resulting bound tends to zero, while the symmetric inner truncation tends to $F_M(s)$. This proves \eqref{eq:asymmetric-limit} for arbitrary outer cutoffs. All the properties of $D_s$ hold for every $s$ on the same pathwise event, so the assertion is simultaneous. Theorem~\ref{thm:stationary} and Slutsky's theorem give \eqref{eq:asymmetric-law}.
\end{proof}

\section{Positive marks and transfer from local limits}\label{sec:transfer}

Positivity gives the monotonicity in the finite identity and the Herglotz property of the limiting transform. A second moment of the atom weights is not required.

\begin{corollary}[Positive marks with a $p$-moment]\label{cor:marks}
Let $X$ be a stationary simple point process of intensity $\lambda>0$. Suppose
\[
\E|X((0,x])-\lambda x|\le h_0(x),\qquad x\ge1,
\]
where $h_0$ is finite, nonnegative and measurable, with $\int_1^\infty h_0(x)x^{-2}\,\mathrm dx<\infty$. Independently mark its points by identically distributed independent positive variables $W_j$ such that
\[
m=\E W_1\in(0,\infty),\qquad \E W_1^p<\infty
\quad\hbox{for some }p>1.
\]
Set $r=\min(p,2)$ and $M=\sum_jW_j\delta_{X_j}$. Then $M$ is stationary, has intensity $K=\lambda m$, and satisfies Theorem~\ref{thm:stationary} with
\begin{equation}\label{eq:marked-discrepancy}
h(x)=m h_0(x)+C x^{1/r}.
\end{equation}
In particular, if $h_0(x)=C_0(1+x^{\alpha_0})$ for some $\alpha_0<1$, then \eqref{eq:polynomial-first-moment} holds for every $\alpha\in[\max(\alpha_0,1/r),1)$. The marks may have infinite variance.
\end{corollary}

\begin{proof}
Independent identically distributed marking preserves stationarity, and conditioning on $X$ gives the intensity $\lambda m$. Let $N=X((0,x])$. Conditional on the point process,
\[
M((0,x])-\lambda mx=m(N-\lambda x)+\sum_{j=1}^N(W_j-m).
\]
The von Bahr--Esseen inequality, with the variance identity at $r=2$, gives
\[
\E\left[\left|\sum_{j=1}^N(W_j-m)\right|^r\,\middle|\,X\right]
\le C_r N\E|W_1-m|^r;
\]
see \cite{BahrEsseen}. Taking expectations and then the $r$th root bounds the first absolute moment of the sum by a constant times $x^{1/r}$, since $\E N=\lambda x$. This proves \eqref{eq:marked-discrepancy} on the positive half-line. Stationarity gives the same bound on the negative half-line. Since $r>1$, the added term is sublinear and integrable after division by $x^2$. The polynomial assertion follows by comparing the exponents.
\end{proof}

\begin{corollary}[Three stationary models]\label{cor:models}
The following transforms exist simultaneously at all nonatoms, are pathwise shift covariant and shift amenable, and have the indicated one-point laws.
\begin{enumerate}[label=(\alph*)]
\item If $X$ is a stationary $\operatorname{Sine}_\beta$ process of intensity $\lambda$, for fixed $\beta,\lambda>0$, then, for every deterministic $s$,
\[
\frac1{\pi\lambda}\lim_{A\to\infty}\sum_{x\in X\cap[-A,A]}\frac1{s-x}
\sim\Cauchy(0,1).
\]
Its centered symmetric truncation error in $L^1$ is $O_{\beta,\lambda}(\sqrt{\log(2+A)}/A)$.
\item In the normalization of the Najnudel--Vir\'ag bead process, the support is $\operatorname{Sine}_\beta$ of intensity $1/(2\pi)$, and the independent weights have distribution $(4/\beta)\operatorname{Gamma}(\beta/2,1)$, with Gamma parametrized by shape and scale. Their mean is two. For the resulting measure $M$, $K=1/\pi$ and
\[
\lim_{A\to\infty}\int_{[-A,A]}\frac{\mathrm dM(x)}{s-x}\sim\Cauchy(0,1).
\]
The centered truncation error is $O_\beta(A^{-1/2})$ in $L^1$.
\item For a Poisson process of intensity $\lambda$, the law in part (a) holds, and the centered truncation error is $O_\lambda(A^{-1/2})$ in $L^1$.
\end{enumerate}
\end{corollary}

\begin{proof}
For a fixed rescaling of $\operatorname{Sine}_\beta$, the logarithmic number-variance estimate \cite[Theorem~1]{Variance} gives
\[
\operatorname{Var}X((0,L])\le C_{\beta,\lambda}\log(2+L).
\]
Cauchy--Schwarz supplies $h(L)=C_{\beta,\lambda}\sqrt{\log(2+L)}$. Theorem~\ref{thm:stationary} proves the law, recovering the marginal in \cite[Theorem~60]{ValkoVirag}; Theorem~\ref{thm:truncation} gives the stated explicit rate. The support intensity and weight normalization in (b) are those of \cite[Theorem~9]{Bead}. The Gamma variables have moments of every positive order, so Corollary~\ref{cor:marks} with $r=2$ supplies $h(L)=O_\beta(\sqrt L)$ for $L\ge1$, proving both assertions. Finally, a Poisson count of an interval of length $L$ has variance $\lambda L$, which yields $h(L)=\sqrt{\lambda L}$ and proves (c).
\end{proof}

Vague convergence of a sequence of point measures alone does not ensure convergence of the corresponding Stieltjes transforms: a point may approach the observation point, and distant points may contribute an unbalanced tail. The next theorem separates these two issues. 

\begin{theorem}[Cauchy transfer]\label{thm:transfer}
Let $M_n$ be positive random weighted point measures, and let $A_n\to\infty$. Suppose that:
\begin{enumerate}[label=(\roman*)]
\item $M_n$ converges vaguely in law to a stationary measure $M$ of intensity $K$ satisfying Theorem~\ref{thm:stationary};
\item the near-pole condition holds:
\begin{equation}\label{eq:transfer-near-pole}
\lim_{\delta\downarrow0}\limsup_{n\to\infty}
\P\bigl(M_n((-\delta,\delta))>0\bigr)=0;
\end{equation}
\item there is a finite nonnegative measurable $h$ on $[1,\infty)$ with
\[
h(x)=o(x),\qquad \int_1^\infty\frac{h(x)}{x^2}\,\mathrm dx<\infty,
\]
such that, uniformly for $1\le x\le A_n$,
\begin{equation}\label{eq:transfer-imbalance}
\E|M_n((0,x])-M_n([-x,0))|\le h(x).
\end{equation}
\end{enumerate}
Then
\begin{equation}\label{eq:transfer-conclusion}
\frac1{\pi K}\int_{0<|x|\le A_n}\frac{\mathrm dM_n(x)}x
\ \Longrightarrow\ \Cauchy(0,1).
\end{equation}
In particular, condition (iii) is satisfied by the bound $C(1+x^\alpha)$ with any $\alpha<1$.
\end{theorem}

\begin{proof}
Put $Q_n(x)=M_n((0,x])-M_n([-x,0))$. For fixed $A>1$ and $n$ large enough that $A_n>A$, the identity \eqref{eq:imbalance-ibp} gives
\begin{equation}\label{eq:transfer-tail-bound}
\E\left|\int_{A<|x|\le A_n}\frac{\mathrm dM_n(x)}x\right|
\le\frac{h(A_n)}{A_n}+\frac{h(A)}A+\int_A^{A_n}\frac{h(x)}{x^2}\,\mathrm dx.
\end{equation}
Here sublinearity controls the imposed endpoint $h(A_n)/A_n$. The upper limit as $n\to\infty$ of the bound tends to zero as $A\to\infty$.

Fix $0<\delta<A$. The limiting measure almost surely has no atoms at the deterministic points $\pm\delta,\pm A$. The compactly supported kernel $x^{-1}\mathbf1_{\{\delta<|x|\le A\}}$ is bounded and continuous away from these four points. The continuous mapping theorem for vague convergence consequently gives
\[
\int_{\delta<|x|\le A}\frac{\mathrm dM_n(x)}x
\ \Longrightarrow\ \int_{\delta<|x|\le A}\frac{\mathrm dM(x)}x.
\]
The difference between the integral with $0<|x|\le A$ and this one vanishes whenever $M_n((-2\delta,2\delta))=0$. The factor two includes possible atoms at $\pm\delta$. Applying \eqref{eq:transfer-near-pole} with $2\delta$ removes the singular neighborhood in probability as first $n\to\infty$ and then $\delta\downarrow0$. For the limiting measure the same removal holds almost surely: its support is locally finite and does not contain zero, so a random neighborhood of zero is empty. Thus, for fixed $A$,
\[
\int_{0<|x|\le A}\frac{\mathrm dM_n(x)}x
\ \Longrightarrow\ \int_{0<|x|\le A}\frac{\mathrm dM(x)}x.
\]
This last conclusion can also be read directly from bounded Lipschitz tests, whose difference is at most twice the probability of a nonempty singular neighborhood.

Now use \eqref{eq:transfer-tail-bound} to let $A\to\infty$. The limiting symmetric integral converges to $-F_M(0)$; Theorem~\ref{thm:truncation} gives convergence of its tails in $L^1$. The converging-together argument therefore identifies the limit as $-F_M(0)/(\pi K)$, which is standard Cauchy by Theorem~\ref{thm:stationary} and symmetry.
\end{proof}

\begin{corollary}[Automatic exclusion of near poles]\label{cor:transfer-counting}
In Theorem~\ref{thm:transfer}, assumption (ii) follows from (i) if the masses of all atoms of all $M_n$ are bounded below by one common constant $c>0$. In particular, it is automatic for integer-valued point measures.
\end{corollary}

\begin{proof}
For fixed $\delta>0$, vague convergence and absence of limiting atoms at $\pm\delta$ imply
$M_n([-\delta,\delta])\Longrightarrow M([-\delta,\delta])$. An atom of $M_n$ in the open interval forces its mass on the closed interval to be at least $c$. The Portmanteau theorem yields
\[
\limsup_{n\to\infty}\P(M_n((-\delta,\delta))>0)
\le\P(M([-\delta,\delta])\ge c).
\]
As $\delta\downarrow0$, the masses on the right decrease almost surely to $M(\{0\})=0$. This proves the claim.
\end{proof}

\begin{remark}\label{rem:transfer-limitations}
The lower bound on atom weights in Corollary~\ref{cor:transfer-counting} is substantive. For example, start with a uniformly translated unit lattice $M$ and set $M_n=M+n^{-1}\delta_{n^{-2}}$. The added atom disappears vaguely and changes the imbalance in \eqref{eq:transfer-imbalance} by at most $n^{-1}$ for $x\ge1$, but contributes $n$ to the transform. Condition \eqref{eq:transfer-near-pole} fails.
\end{remark}

\section{Relation with the Aizenman--Warzel theory}\label{sec:analytic}

The qualitative Cauchy law belongs to a general scalar Herglotz--Pick theory. We recall the relevant statement to distinguish it from the quantitative construction above.

\begin{proposition}[Aizenman--Warzel]\label{prop:aw}
Let $G$ be a Herglotz--Pick function with real boundary values almost everywhere. Suppose its boundary function is shift amenable and
\[
 G(\ii\eta)\longrightarrow a+\ii b\qquad(\eta\to\infty),
 \qquad b\ge0.
\]
Then its boundary shift distribution is $\Cauchy(a,b)$, with $b=0$ interpreted as the point mass at $a$.
\end{proposition}

This is \cite[Theorem~2.3 and Lemma~2.5]{AW}. Under our discrepancy condition, $G_M=-F_M$ has real boundary values and $G_M(\ii\eta)\to\ii\pi K$. Birkhoff's theorem, applied to bounded observables integrated over a unit interval, gives shift amenability. Thus the principle recovers the stationary law and, componentwise, its scale mixtures. Our pathwise hypotheses also specialize \cite[Theorem~4.1]{AW} to the reference measure $K\dd x$.

Periodicity gives shift amenability directly. Related formulations are \cite[Proposition~2.6]{AW} and, for circular $\beta$-ensembles, \cite[Remark~3.2(1)]{Forrester}. The finite conditional identity supplies our approximating law; periodization and truncation estimates control transfer and field convergence. For zeta zeros, the principle applies after horizontal projection, whereas comparison with the actual logarithmic derivative requires Section~\ref{sec:correction}. The Riemann-hypothesis case is due to Sodin \cite[Section~3.4.5, equation~(55)]{Sodin}.

\section{Projected ordinates of zeta zeros}\label{sec:projected}

In this section and the next one, we apply the previous results to the setting of the Riemann zeta function. 
Let $\rho=\beta+\ii\gamma$ range over the distinct nontrivial zeros of $\zeta$, and let $m(\rho)$ denote the multiplicity of $\rho$. Throughout this section, for $T \geq 100$, let $L=\log T$ and let $t_T$ be uniform on
\[
 I_T=[T/L,T].
\]
We consider the random point measure
\begin{equation}\label{eq:zeta-projected-measure}
 u_{\rho,T}=\frac{(\gamma-t_T)L}{2\pi},
 \qquad M_T=\sum_{\rho}m(\rho)\delta_{u_{\rho,T}}.
\end{equation}
Zeros with the same ordinate form a single atom whose mass is their total multiplicity. In particular, no simplicity assumption is made.

\begin{theorem}[Unconditional projected-ordinate law]\label{thm:projected}
Without assuming the Riemann hypothesis,
\begin{equation}\label{eq:zeta-projected-law}
 P_T:=-\frac1\pi\sum_{0<|u_{\rho,T}|\le L^2}
          \frac{m(\rho)}{u_{\rho,T}}
 \ \Longrightarrow\ \Cauchy(0,1).
\end{equation}
\end{theorem}

We first record the counting estimates used in the proof. Let $N(v)$ count the nontrivial zeros with ordinates in $(0,v]$, with multiplicity, and extend it oddly to negative $v$. The Riemann--von Mangoldt formula gives
\begin{equation}\label{eq:zeta-counting-formula}
 N(v)=\Phi(v)+S(v)+O(1),
 \qquad
 \Phi(v)=\frac1{2\pi}\int_0^v\log^+\frac{|y|}{2\pi}\,\dd y,
\end{equation}
where $S(v)=O(\log(2+|v|))$; endpoint conventions at zero ordinates have no effect on any integral below.

We shall use the following two-sided increment estimate, uniformly for $T\ge100$ and $h\in\R$:
\begin{equation}\label{eq:zeta-fujii-global}
 \int_0^T|S(t+h)-S(t)|^2\,\dd t
 \ll T\log^2\bigl(2+|h|\log(2+T)\bigr).
\end{equation}
Here is a derivation which specifies the range of the input. Fujii's estimate, in its corrected small-shift range, gives
\[
 \int_X^{2X}|S(t+h)-S(t)|^2\,\dd t
 \ll X\log(2+h\log X),\qquad 0<h\le1,
\]
for $X\ge10$; see \cite{Fujii1974,Fujii1981,Fujii1999}, or \cite[Theorem~2.2 and its accompanying footnote]{MaplesRodgers}. For $1<h\le X/2$, Selberg's mean-square estimate
\[
 \int_X^{3X}|S(t)|^2\,\dd t\ll X\log\log X
\]
(see \cite{Selberg,Titchmarsh}) gives the same upper bound, since $\log\log X\ll\log(2+h\log X)$. A change of variables and a cover by a bounded number of intervals $[Y,2Y]$, with $Y\asymp X$, give the corresponding bound for $-X/3\le h<0$. If $|h|>X/3$, the pointwise estimate for $S$ instead gives
\[
 \int_X^{2X}|S(t+h)-S(t)|^2\,\dd t
 \ll X\log^2\bigl(2+|h|\log(2+X)\bigr).
\]
Thus the squared logarithm is needed only for large shifts. Summing over the dyadic intervals making up $[0,T]$, and absorbing the final bounded interval, proves \eqref{eq:zeta-fujii-global}. All these estimates count zeros throughout the critical strip and are unconditional.

\begin{lemma}[Counting and tail estimates]\label{lem:counting}
Uniformly for $|x|\le L^2$,
\begin{equation}\label{eq:zeta-counting-second-moment}
 \E|N_{M_T}(x)-x|^2\ll 1+|x|^{4/3}.
\end{equation}
For every fixed $a<b$,
\begin{equation}\label{eq:zeta-counting-mean}
 \E M_T((a,b])\longrightarrow b-a.
\end{equation}
Moreover, for every fixed $0<\delta<1$,
\begin{equation}\label{eq:zeta-near-pole}
 \P\bigl(M_T((-\delta,\delta))>0\bigr)\ll\delta+o(1),
\end{equation}
and, uniformly for $1\le A<B\le L^2$,
\begin{equation}\label{eq:zeta-projected-tail}
 \E\left|\int_{A<|u|\le B}\frac{\dd M_T(u)}{u}\right|
 \ll A^{-1/3}.
\end{equation}
\end{lemma}

\begin{proof}
For $x\ge0$, almost surely,
\[
 N_{M_T}(x)=N\left(t_T+\frac{2\pi x}{L}\right)-N(t_T),
\]
and the analogous identity holds for the signed count when $x<0$. The smooth part of \eqref{eq:zeta-counting-formula} satisfies, uniformly in $t\in I_T$ and $|x|\le L^2$,
\[
 \Phi\left(t+\frac{2\pi x}{L}\right)-\Phi(t)
 =x+O\left(\frac{|x|\log L}{L}\right).
\]
Here $|x|\log L/L\le |x|^{2/3}\log L/L^{1/3}$, since $|x|\le L^2$, so the error is $O(|x|^{2/3})$. Since $|I_T|\asymp T$, \eqref{eq:zeta-fujii-global}, with $h=2\pi x/L$, bounds the mean square of the $S$-increment by $O(\log^2(2+|x|))$. This proves \eqref{eq:zeta-counting-second-moment}. Any exponent $\alpha\in(1/2,1)$ could similarly replace $2/3$.

For the mean, write $V(v)=N(v)-\Phi(v)$, so that $V(v)=O(\log(2+|v|))$. If $h=O(1/L)$ is fixed at the microscopic scale, translation of the integral gives
\[
 \left|\int_{I_T}\{V(t+h)-V(t)\}\,\dd t\right|
 \ll |h|\log T.
\]
Indeed, the two integrals cancel on the intersection of $I_T$ and $I_T+h$, leaving two intervals of length $|h|$. Dividing by $|I_T|$ and applying the preceding smooth estimate shows that $\E N_{M_T}(x)\to x$ for each fixed $x$. Taking the difference at $a$ and $b$ proves \eqref{eq:zeta-counting-mean}. Since every nonzero atom of $M_T$ has mass at least one, Markov's inequality gives \eqref{eq:zeta-near-pole}.

For $x\ge0$, put
\[
 Q_T(x)=M_T((0,x])-M_T([-x,0)).
\]
The Lebesgue reference terms cancel, and \eqref{eq:zeta-counting-second-moment} gives
\[
 \E|Q_T(x)|\ll 1+x^{2/3},\qquad x\le L^2.
\]
Stieltjes integration by parts yields, almost surely for deterministic $A$ and $B$,
\[
 \int_{A<|u|\le B}\frac{\dd M_T(u)}u
 =\frac{Q_T(B)}B-\frac{Q_T(A)}A
      +\int_A^B\frac{Q_T(x)}{x^2}\,\dd x.
\]
Taking expectations proves \eqref{eq:zeta-projected-tail}.
\end{proof}

\begin{proof}[Proof of Theorem~\ref{thm:projected}]
For every fixed $A$, Lemma~\ref{lem:counting} gives
\[
 \sup_{T\text{ sufficiently large}}\E M_T([-A,A])<\infty.
\]
Markov's inequality on an increasing sequence of compact intervals gives tightness for vague convergence of locally finite measures. Along a convergent subsequence, write $M_T\Longrightarrow M$. Since the $M_T$ are integer-valued point measures, their locally finite vague limit is also an integer-valued point measure.

With our convention $\theta_hM(B)=M(B+h)$, translating $M_T$ by a fixed $h$ amounts to replacing $t_T$ by $t_T+2\pi h/L$. The total variation distance between the shifted uniform distribution and the original distribution on $I_T$ is $O(|h|/(TL))$. Every subsequential limit $M$ is therefore stationary.

For each bounded interval $J$, the second moments of $M_T(J)$ are uniformly bounded, by writing its mass as a difference of signed counts based at zero and using \eqref{eq:zeta-counting-second-moment}. Consequently $M_T(f)$ is uniformly integrable for every continuous compactly supported $f$. Approximation by step functions in \eqref{eq:zeta-counting-mean} gives
\[
 \E M_T(f)\longrightarrow\int_\R f(x)\,\dd x.
\]
Vague convergence and uniform integrability imply $\E M(f)=\int f(x)\,\dd x$. Thus $M$ has intensity one, and it almost surely has no atom at any prescribed point. For each fixed $x$, convergence of the interval masses and Fatou's lemma now give
\[
 \E|N_M(x)-x|^2\ll1+|x|^{4/3}.
\]
Corollary~\ref{cor:variance} applies, and hence
\[
 -\frac1\pi\pv\int_\R\frac{\dd M(u)}u
 \sim\Cauchy(0,1).
\]
Finally, \eqref{eq:zeta-near-pole} controls the pole at zero, while Cauchy--Schwarz in \eqref{eq:zeta-counting-second-moment} supplies the discrepancy bound $h(x)=C(1+x^{2/3})$ uniformly up to $L^2$. Theorem~\ref{thm:transfer}, with cutoff $L^2$, therefore proves \eqref{eq:zeta-projected-law} along the chosen subsequence. Every subsequence admits a further subsequence with this same limit, which proves the assertion for the full family. The same proof permits any deterministic cutoff $b_T\to\infty$ with $b_T\le L^2$; the choice $L^2$ also gives an $O(L^{-1})$ product tail in Appendix~\ref{app:factorization}.
\end{proof}

\section{The off-critical correction}\label{sec:correction}

We first derive the correction by pairing functional-equation partners. Selberg's estimate then relates the horizontal moment conditions, and fractional-moment and weak-type bounds control the correction.

We retain the notation of Section~\ref{sec:projected}. For a zero $\rho=\beta+\ii\gamma$, define its complex microscopic position by
\begin{equation}\label{eq:zeta-complex-position}
 z_{\rho,T}=\frac{L}{2\pi\ii}\left(\rho-\frac12-\ii t_T\right)
 =u_{\rho,T}-\ii\eta_{\rho,T},
 \qquad \eta_{\rho,T}=\frac{(\beta-1/2)L}{2\pi}.
\end{equation}
We shall compare $P_T$ with
\begin{equation}\label{eq:zeta-actual-statistic}
 A_T=\frac{2\ii}{L}\frac{\zeta'}{\zeta}\left(\frac12+\ii t_T\right)+\ii.
\end{equation}
The continuously distributed height avoids all zero ordinates almost surely.

\begin{lemma}[Local logarithmic derivative]\label{lem:local-log}
One has
\begin{equation}\label{eq:zeta-local-log}
 A_T=-\frac1\pi\sum_{|u_{\rho,T}|\le L^2}
                   \frac{m(\rho)}{z_{\rho,T}}+O\left(\frac{\log L}{L}\right).
\end{equation}
The error is uniform for deterministic $t\in I_T$ such that $\zeta(1/2+\ii t)\ne0$. More precisely, its remainder $r_T(t)$ satisfies
\begin{equation}\label{eq:zeta-local-log-refined}
 |r_T(t)|\le C\frac{1+\log(T/t)}L,
 \qquad \E|r_T(t_T)|\le\frac CL.
\end{equation}
\end{lemma}

\begin{proof}
Proposition~\ref{prop:local-factorization} gives the factorization and the pointwise remainder bound after differentiation at zero. The average follows from
\[
 \E\log(T/t_T)=1-\frac{\log L}{L-1}\le1.
\]
The factorization retains the functional-equation partners as complex zeros, so no assumption on their real parts is used.
\end{proof}

\begin{proposition}[Off-critical correction formula]\label{prop:correction}
Almost surely,
\begin{align}
 A_T-P_T
 &=-\frac4L\sum_{\substack{\rho=\beta+\ii\gamma,\ \beta>1/2\\
                         |\gamma-t_T|\le2\pi L}}
 \frac{m(\rho)(\beta-1/2)^2}
 {(t_T-\gamma)((t_T-\gamma)^2+(\beta-1/2)^2)}
 +O\left(\frac{\log L}{L}\right)\label{eq:zeta-correction}\\
 &=\frac2\pi\sum_{\substack{\beta>1/2\\|u_{\rho,T}|\le L^2}}
 \frac{m(\rho)\eta_{\rho,T}^2}
 {u_{\rho,T}(u_{\rho,T}^2+\eta_{\rho,T}^2)}
 +O\left(\frac{\log L}{L}\right).\label{eq:zeta-correction-micro}
\end{align}
Both errors equal $r_T(t_T)$ and satisfy \eqref{eq:zeta-local-log-refined}.
\end{proposition}

\begin{proof}
The functional equation and complex conjugation pair $1/2+a+\ii\gamma$ with $1/2-a+\ii\gamma$, with equal multiplicity. They have the same ordinate and hence belong to the truncation window together. Their positions in \eqref{eq:zeta-complex-position} are $u-\ii\eta$ and $u+\ii\eta$, and their contribution in \eqref{eq:zeta-local-log} is
\[
 -\frac1\pi\left(\frac1{u-\ii\eta}+\frac1{u+\ii\eta}\right)
 =-\frac2\pi\frac{u}{u^2+\eta^2},
\]
whereas their projected contribution is $-2/(\pi u)$. Subtraction gives $2\eta^2/[\pi u(u^2+\eta^2)]$. The zeros on the critical line contribute no difference. Summing proves \eqref{eq:zeta-correction-micro}; substituting the definitions of $u$ and $\eta$ gives \eqref{eq:zeta-correction}.
\end{proof}

\begin{lemma}[Off-window correction tail]\label{lem:off-tail}
Uniformly for $t\in I_T$,
\begin{equation}\label{eq:zeta-off-window-tail}
 \frac4L\sum_{\substack{\rho=\beta+\ii\gamma,\ \beta>1/2\\
                      0<\gamma\le2T,\ |\gamma-t|>2\pi L}}
 \frac{m(\rho)(\beta-1/2)^2}
 {|t-\gamma|((t-\gamma)^2+(\beta-1/2)^2)}
 \ll L^{-2}.
\end{equation}
\end{lemma}

\begin{proof}
Since $0<\beta-1/2<1/2$, the left-hand side is at most
\[
 \frac1L\sum_{\substack{\beta>1/2,\ 0<\gamma\le2T\\|\gamma-t|>2\pi L}}
              \frac{m(\rho)}{|t-\gamma|^3}.
\]
The Riemann--von Mangoldt formula gives the uniform unit-interval bound
\begin{equation}\label{eq:zeta-unit-count}
 \sum_{y<\gamma\le y+1}m(\rho)\ll\log(2+y),\qquad y\ge0.
\end{equation}
Grouping the ordinates according to their distance from $t$ therefore bounds the preceding expression by
\[
 \frac{C\log(2+T)}L\sum_{k\ge2\pi L-1}k^{-3}\ll L^{-2}.
\]
\end{proof}

We next give equivalent forms of the horizontal-collapse hypothesis. Put
\begin{align}
 N_{\mathrm{off}}(X)&=\sum_{\substack{0<\gamma\le X\\\beta\ne1/2}}m(\rho),
 \label{eq:zeta-off-count}\\
 D_\varepsilon(T)&=\sum_{\substack{0<\gamma\le2T\\
                         |\beta-1/2|\ge\varepsilon/L}}m(\rho),
 \qquad\varepsilon>0,\label{eq:zeta-micro-count}\\
 H(T)&=\sum_{\substack{0<\gamma\le2T\\\beta>1/2}}
                         m(\rho)(\beta-1/2),\label{eq:zeta-horizontal-moment}\\
 \mu_T&=\frac1{TL}\sum_{\substack{0<\gamma\le2T\\\beta>1/2}}
                         m(\rho)\delta_{L(\beta-1/2)}.
 \label{eq:zeta-horizontal-measure}
\end{align}
The mass of $\mu_T$ is uniformly bounded, but it need not converge to zero: its atoms may instead approach the origin.

\begin{lemma}[Exponential horizontal tail]\label{lem:selberg}
There are constants $c,C>0$ such that, uniformly for $T\ge3$ and $b\ge0$,
\begin{equation}\label{eq:zeta-selberg-tail}
 \mu_T([b,\infty))\le Ce^{-cb}.
\end{equation}
Consequently, for every fixed $q>0$,
\begin{equation}\label{eq:zeta-horizontal-moment-bound}
 \sum_{\substack{0<\gamma\le2T\\\beta>1/2}}
               m(\rho)(\beta-1/2)^q
 \ll_q T(\log T)^{1-q}.
\end{equation}
In particular, $H(T)\ll T$.
\end{lemma}

\begin{proof}
Selberg's zero-density estimate \cite{Selberg}, in the form recalled in \cite[Theorem~2.4]{MaplesRodgers}, gives an absolute constant $C$ such that
\[
 N(\sigma,X):=\sum_{\substack{0<\gamma\le X\\\beta\ge\sigma}}m(\rho)
 \le CX\log X\,e^{-(\sigma-1/2)\log X/4},
 \qquad \frac12\le\sigma\le1.
\]
Taking $X=2T$ and $\sigma=1/2+b/L$ proves \eqref{eq:zeta-selberg-tail}, with $c=1/4$, for $0\le b\le L/2$, since $1\le\log(2T)/L\le1+\log2/\log3$; for larger $b$ the left-hand side vanishes. Layer-cake integration now gives
\[
 \int_0^\infty b^q\,\dd\mu_T(b)
 =q\int_0^\infty b^{q-1}\mu_T((b,\infty))\,\dd b
 \le Cq\int_0^\infty b^{q-1}e^{-cb}\,\dd b<\infty,
\]
uniformly in $T$. This is \eqref{eq:zeta-horizontal-moment-bound}.
\end{proof}

\begin{samepage}
\begin{proposition}[Equivalent forms of microscopic horizontal collapse]\label{prop:collapse-equivalence}
The following assertions are equivalent:
\begin{enumerate}[label=(\roman*)]
\item For every fixed $\varepsilon>0$, $D_\varepsilon(T)=o(T\log T)$.
\item $H(T)=o(T)$.
\item For some fixed $q>0$,
\begin{equation}\label{eq:zeta-one-moment-collapse}
 \sum_{\substack{0<\gamma\le2T\\\beta>1/2}}
 m(\rho)(\beta-1/2)^q=o\bigl(T(\log T)^{1-q}\bigr).
\end{equation}
\item Equation~\eqref{eq:zeta-one-moment-collapse} holds for every fixed $q>0$.
\end{enumerate}
\end{proposition}
\end{samepage}

\begin{proof}
Functional-equation symmetry gives, for every $\varepsilon>0$,
\[
 \mu_T([\varepsilon,\infty))=\frac{D_\varepsilon(T)}{2TL}.
\]
Thus (i) is exactly convergence to zero of each of these tails. Under (i), Lemma~\ref{lem:selberg} and dominated convergence in the layer-cake formula give
\[
 \int b^q\,\dd\mu_T(b)\longrightarrow0
 \qquad(q>0),
\]
because $b\mapsto Cq b^{q-1}e^{-cb}$ is integrable. This proves (iv). Conversely, if this integral tends to zero for any $q>0$, Markov's inequality yields
\[
 \frac{D_\varepsilon(T)}{2TL}
 \le\varepsilon^{-q}\int b^q\,\dd\mu_T(b)\longrightarrow0,
\]
which proves (iii)$\Rightarrow$(i). Finally,
\[
 \int b\,\dd\mu_T(b)=\frac{H(T)}T,
\]
so (ii) is the case $q=1$. The remaining implications are immediate.
\end{proof}

For comparison with fractional moments, define, for $0<p<1$,
\begin{equation}\label{eq:zeta-fractional-horizontal-moment}
 W_p(X)=\sum_{\substack{0<\gamma\le X\\\beta>1/2}}
                 m(\rho)(\beta-1/2)^{1-p}.
\end{equation}
Proposition~\ref{prop:collapse-equivalence} shows that, for any fixed $p\in(0,1)$, the condition $W_p(2T)=o(TL^p)$ is equivalent to microscopic horizontal collapse.

Now, put 
\begin{equation}\label{eq:zeta-full-correction}
 g_a(x)=\frac{a^2}{x(x^2+a^2)},\qquad
 C_T(t)=-\frac4L\sum_{\substack{0<\gamma\le2T\\\beta>1/2}}
                   m(\rho)g_{\beta-1/2}(t-\gamma).
\end{equation}
For all sufficiently large $T$, the ordinate window in Proposition~\ref{prop:correction} is contained in $(0,2T)$. That proposition and Lemma~\ref{lem:off-tail} therefore give
\begin{equation}\label{eq:zeta-full-comparison}
 A_T-P_T=C_T(t_T)+E_T(t_T),\qquad
 |E_T(t)|\le C\frac{1+\log(T/t)}L,
\end{equation}
away from zero ordinates. In particular, $E_T(t)=O(\log L/L)$ uniformly on $I_T$, while $\E|E_T(t_T)|\le C/L$ by \eqref{eq:zeta-local-log-refined}.

\begin{proposition}[Fractional-moment and weak-type estimates]\label{prop:fractional}
For every fixed $p\in(1/3,1)$,
\begin{equation}\label{eq:zeta-fractional-estimate}
 \E|C_T(t_T)|^p\le\frac{C_p}{TL^p}W_p(2T).
\end{equation}
In particular,
\begin{equation}\label{eq:zeta-half-moment-estimate}
 \E|C_T(t_T)|^{1/2}\le C\sqrt{\frac{H(T)}T}.
\end{equation}
Also, for every $\lambda>0$,
\begin{equation}\label{eq:zeta-weak-type-estimate}
 \P\bigl(|C_T(t_T)|>\lambda\bigr)
 \le\frac{C}{\lambda}\frac{N_{\mathrm{off}}(2T)}{T\log T}.
\end{equation}
\end{proposition}

\begin{proof}
A change of variables gives, for $a>0$,
\begin{equation}\label{eq:zeta-kernel-integral}
 \int_\R|g_a(x)|^p\,\dd x
 =c_p a^{1-p},\qquad
 c_p=\int_\R\frac{\dd y}{|y|^p(1+y^2)^p}<\infty.
\end{equation}
The integrand behaves as $|y|^{-p}$ at zero and $|y|^{-3p}$ at infinity, giving precisely $p<1$ and $3p>1$. Since $p<1$, the inequality $|\sum_jx_j|^p\le\sum_j|x_j|^p$ applies, with multiplicities counted as repeated summands. Averaging over $I_T$ and then extending each integral to $\R$ proves
\[
 \E|C_T(t_T)|^p
 \le\frac{4^pc_p}{|I_T|L^p}
       \sum_{\substack{0<\gamma\le2T\\\beta>1/2}}
                    m(\rho)(\beta-1/2)^{1-p},
\]
which is \eqref{eq:zeta-fractional-estimate}. At $p=1/2$, Cauchy--Schwarz and the Riemann--von Mangoldt formula give
\[
 W_{1/2}(2T)
 \le\left(\sum_{\substack{0<\gamma\le2T\\\beta>1/2}}m(\rho)\right)^{1/2}
       H(T)^{1/2}
 \ll (TLH(T))^{1/2}.
\]
This proves \eqref{eq:zeta-half-moment-estimate}.

For the weak-type bound, write
\[
 P_a(x)=\frac1\pi\frac{a}{x^2+a^2},\qquad
 \mathcal H\nu(t)=\frac1\pi\pv\int_\R\frac{\dd\nu(y)}{t-y}.
\]
The Hilbert transform of the Poisson kernel satisfies
\[
 \mathcal H\bigl(P_a(\cdot-\gamma)\,\dd y-\delta_\gamma\bigr)(t)
 =-\frac1\pi g_a(t-\gamma).
\]
Form the finite signed measure
\[
 \nu_T=\sum_{\substack{0<\gamma\le2T\\\beta>1/2}}
 m(\rho)\bigl(P_{\beta-1/2}(\cdot-\gamma)\,\dd y-\delta_\gamma\bigr).
\]
Each summand has total variation at most twice its multiplicity. Since off-line zeros occur in horizontal pairs,
\[
 \|\nu_T\|_{\mathrm{TV}}\le N_{\mathrm{off}}(2T),\qquad
 C_T(t)=\frac{4\pi}{L}\mathcal H\nu_T(t).
\]
To apply the classical $L^1$ weak-$(1,1)$ inequality \cite[Chapter~II]{Stein}, first convolve $\nu_T$ with $P_\varepsilon$, $\varepsilon>0$. This gives the integrable function
\[
 (\nu_T*P_\varepsilon)(x)
 =\sum_{\substack{0<\gamma\le2T\\\beta>1/2}}
 m(\rho)\{P_{\beta-1/2+\varepsilon}(x-\gamma)-P_\varepsilon(x-\gamma)\},
\]
whose $L^1$ norm is at most $N_{\mathrm{off}}(2T)$. Since $\mathcal HP_a(x)=x/[\pi(x^2+a^2)]$, its Hilbert transform converges pointwise away from the finitely many ordinates to $\mathcal H\nu_T$. The $L^1$ inequality and Fatou's lemma therefore yield
\[
 \frac1{|I_T|}\operatorname{meas}
 \left\{t\in I_T:\left|\frac{4\pi}{L}\mathcal H\nu_T(t)\right|>\lambda\right\}
 \ll\frac{N_{\mathrm{off}}(2T)}{\lambda TL},
\]
which is \eqref{eq:zeta-weak-type-estimate}.
\end{proof}

\begin{theorem}[A horizontal-moment replacement for the Riemann hypothesis]\label{thm:zeta}
Assume any of the equivalent conditions of Proposition~\ref{prop:collapse-equivalence}. If $\omega$ is uniform on $[0,1]$, then
\begin{equation}\label{eq:zeta-actual-cauchy}
 \frac{2\ii}{\log T}\frac{\zeta'}{\zeta}
       \left(\frac12+\ii T\omega\right)+\ii
 \ \Longrightarrow\ \Cauchy(0,1).
\end{equation}
More quantitatively, for the height $t_T$ uniform on $I_T$,
\begin{equation}\label{eq:zeta-bl-comparison}
 \dBL\bigl(\Law(A_T),\Law(P_T)\bigr)
 \le C\sqrt{\frac{H(T)}T}+\frac C{\log T}.
\end{equation}
Here the bounded-Lipschitz distance is taken on $\C$, with the real law of $P_T$ viewed as a law on $\C$. This is a comparison bound with $P_T$, whose Cauchy convergence in Theorem~\ref{thm:projected} is qualitative.
\end{theorem}

\begin{proof}
By Proposition~\ref{prop:collapse-equivalence}, it suffices to assume $H(T)=o(T)$. Equation~\eqref{eq:zeta-half-moment-estimate} and Markov's inequality imply $C_T(t_T)\to0$ in probability. Equation~\eqref{eq:zeta-full-comparison} and Theorem~\ref{thm:projected} then give $A_T\Longrightarrow\Cauchy(0,1)$.

For the quantitative assertion, if $f:\C\to\R$ has supremum norm and Lipschitz constant at most one, then
\[
 |f(z+w)-f(z)|\le\min\{2,|w|\}\le\sqrt2\,|w|^{1/2}.
\]
Apply this to $C_T(t_T)$ and use the Lipschitz bound separately for $E_T(t_T)$ in \eqref{eq:zeta-full-comparison}. Equations \eqref{eq:zeta-half-moment-estimate} and $\E|E_T(t_T)|\le C/L$ prove \eqref{eq:zeta-bl-comparison} after taking the supremum over $f$.

Finally, $\P(T\omega<T/L)=1/L$, and conditionally on its complement $T\omega$ is uniform on $I_T$. The total variation distance between these two height distributions is $1/L$, and the same upper bound holds after any measurable pushforward. This transfers the convergence to \eqref{eq:zeta-actual-cauchy}.
\end{proof}

\begin{corollary}[A density-one replacement for the Riemann hypothesis]\label{cor:density-one}
If $N_{\mathrm{off}}(T)=o(T\log T)$, then \eqref{eq:zeta-actual-cauchy} holds.
\end{corollary}

\begin{proof}
For every $\varepsilon>0$, one has $D_\varepsilon(T)\le N_{\mathrm{off}}(2T)=o(T\log T)$. Theorem~\ref{thm:zeta} therefore applies.
\end{proof}

The microscopic condition permits a density-one set of zeros to have nonzero horizontal displacements, provided those displacements are $o(1/\log T)$. It is thus formally weaker than the hypothesis of Corollary~\ref{cor:density-one}. The estimates of Proposition~\ref{prop:fractional} distinguish the information supplied by the two conditions: the fractional-moment estimate uses the sizes of the displacements, whereas the weak-type estimate only uses the number of off-line zeros.

\begin{corollary}[A local critical-line criterion]\label{cor:line-gue}
Suppose that the critical-line process
\begin{equation}\label{eq:zeta-line-process}
 Z^{\mathrm{line}}_{t_T,T}
 =\sum_{\substack{\rho=\beta+\ii\gamma\\\beta=1/2}}
       m(\rho)\delta_{(t_T-\gamma)L/(2\pi)}
\end{equation}
converges vaguely in law to a point process $Z$ with mean measure equal to Lebesgue measure. Then
\[
 N_{\mathrm{off}}(T)=o(T\log T),
\]
and \eqref{eq:zeta-actual-cauchy} follows. In particular, critical-line strong GUE, namely convergence to the sine-kernel process of intensity one, suffices.
\end{corollary}

\begin{proof}
The mass of the process in \eqref{eq:zeta-line-process} on this interval is bounded by $M_T([-1,1])$, whose second moments are uniformly bounded by Lemma~\ref{lem:counting}. The mean measure assumption gives $\E Z([-1,1])=2$ and almost surely no atom at either endpoint. Vague convergence and uniform integrability therefore imply
\[
 \E Z^{\mathrm{line}}_{t_T,T}([-1,1])\longrightarrow2.
\]
On the other hand, \eqref{eq:zeta-counting-mean} gives $\E M_T([-1,1])\to2$. If $M_T^{\mathrm{off}}$ denotes the restriction of the sum defining $M_T$ to off-line zeros, subtraction gives
\[
 \E M_T^{\mathrm{off}}([-1,1])\longrightarrow0.
\]
Every off-line zero with
\[
 T/L+2\pi/L\le\gamma\le T-2\pi/L
\]
is counted in $M_T^{\mathrm{off}}([-1,1])$ for a set of heights in $I_T$ of length $4\pi/L$. Hence the number of off-line zeros in this interior interval, with multiplicity, is at most
\[
 \frac{|I_T|L}{4\pi}\,\E M_T^{\mathrm{off}}([-1,1])=o(TL).
\]
The total number of zeros below $T/L+2\pi/L$ is $O(T)$ by the Riemann--von Mangoldt formula, and the upper endpoint interval has $O(L)$ zeros by \eqref{eq:zeta-unit-count}. It follows that $N_{\mathrm{off}}(T)=o(TL)$. Corollary~\ref{cor:density-one} completes the proof.
\end{proof}

The proof uses only $\E Z^{\mathrm{line}}_{t_T,T}([-1,1])\to2$. Conversely, density one implies this condition, since
\[
 \E M_T^{\mathrm{off}}([-1,1])
 \le\frac{4\pi}{|I_T|L}N_{\mathrm{off}}(2T)=o(1).
\]

We finish with the field-level statement. Let $\mathcal S$ be the sine-kernel point process of intensity one, and define the stochastic zeta function by
\begin{equation}\label{eq:zeta-stochastic-product}
 \xi_\infty(s)=e^{\ii\pi s}
       \lim_{A\to\infty}\prod_{\substack{x\in\mathcal S\\|x|\le A}}
                         \left(1-\frac{s}{x}\right).
\end{equation}
The limit exists almost surely, locally uniformly in $s$. Indeed, the sine-process counting estimates used in Corollary~\ref{cor:models} give convergence of $\pv\sum_{x\in\mathcal S}x^{-1}$, while finite intensity gives $\sum_{|x|>1}x^{-2}<\infty$ almost surely. Expanding the logarithm on the tail then gives locally uniform convergence of the product. This is the normalization of \cite[Theorem~1.5]{CNN}.

For $t\in I_T$ outside the zero ordinates, set
\begin{equation}\label{eq:zeta-pole-cancelled-field}
 w_{t,T}(s)=\frac12+\ii\left(t-\frac{2\pi s}{L}\right),
 \qquad
 \widehat F_{t,T}(s)=
 \frac{(w_{t,T}(s)-1)\zeta(w_{t,T}(s))}
      {(w_{t,T}(0)-1)\zeta(w_{t,T}(0))}.
\end{equation}
The factor $w_{t,T}(s)-1$ cancels the pole of $\zeta$, so $\widehat F_{t,T}$ is entire and equals one at zero. Its zeros arising from nontrivial zeros $\rho$ of $\zeta$ are at $s=-z_{\rho,T}$; on the critical line this is the reflected coordinate $(t-\gamma)L/(2\pi)$ in \eqref{eq:zeta-line-process}. The sine process is invariant in law under reflection.

\begin{corollary}[The stochastic-zeta field and its Cauchy coordinate]\label{cor:stochastic-zeta}
Assume
\[
 \widehat F_{t_T,T}\Longrightarrow\xi_\infty
\]
in the space $\Hol(\C)$ of entire functions with the compact-open topology. Then
\begin{equation}\label{eq:zeta-joint-field-law}
 \bigl(\widehat F_{t_T,T},A_T\bigr)
 \Longrightarrow
 \left(\xi_\infty,\frac{\ii\pi-\xi_\infty'(0)}\pi\right)
 \qquad\text{in }\Hol(\C)\times\C.
\end{equation}
The second coordinate is standard Cauchy. The same joint convergence holds with $t_T$ replaced by $T\omega$, where $\omega$ is uniform on $[0,1]$.
\end{corollary}

\begin{proof}
Put $w_0=1/2+\ii t_T$. Differentiation of \eqref{eq:zeta-pole-cancelled-field} gives the exact identity
\[
 \widehat F_{t_T,T}'(0)
 =-\frac{2\pi\ii}{L}\left\{
                 \frac{\zeta'}{\zeta}(w_0)+\frac1{w_0-1}\right\}.
\]
Consequently,
\begin{equation}\label{eq:zeta-field-coordinate}
 A_T=\ii-\frac{\widehat F_{t_T,T}'(0)}\pi
             -\frac{2\ii}{L(w_0-1)}.
\end{equation}
The final term is uniformly bounded in modulus by $2/(Lt_T)\le2/T$. The map
\[
 f\longmapsto\bigl(f,\ii-f'(0)/\pi\bigr)
\]
is continuous for the stated topologies. The continuous-mapping theorem and \eqref{eq:zeta-field-coordinate} prove \eqref{eq:zeta-joint-field-law}.

By locally uniform convergence in \eqref{eq:zeta-stochastic-product},
\[
 \xi_\infty'(0)=\ii\pi-\pv\sum_{x\in\mathcal S}\frac1x.
\]
The second coordinate in \eqref{eq:zeta-joint-field-law} is therefore $\pi^{-1}\pv\sum_{x\in\mathcal S}x^{-1}$, which is standard Cauchy by Corollary~\ref{cor:models} and symmetry. Finally, the total variation distance $1/L$ between the two height distributions bounds the distance between their pushforwards under the joint statistic in \eqref{eq:zeta-joint-field-law}. This proves the last assertion.
\end{proof}

\begin{remark}[Scope of the conclusions]\label{rem:zeta-scope}
Theorem~\ref{thm:projected} is unconditional. The horizontal-collapse conditions of Proposition~\ref{prop:collapse-equivalence} are hypotheses on the zeros, and Theorem~\ref{thm:zeta} does not assert an unconditional Cauchy limit for the actual logarithmic derivative. The critical-line and field-convergence assumptions in Corollaries~\ref{cor:line-gue} and \ref{cor:stochastic-zeta} give separate sufficient routes, with complete proofs here. The relation between zero-process convergence and convergence of normalized holomorphic functions is developed further in the companion paper \cite{Companion}.

The Cauchy coordinate in \eqref{eq:zeta-joint-field-law} is a measurable functional of $\xi_\infty$. A scalar Cauchy marginal cannot determine its zero correlations: the stationary results above give the same normalized one-point law for processes as different as Poisson and sine processes. Likewise, negligibility of the correction in \eqref{eq:zeta-correction} is sufficient for the comparison argument, but is not asserted to be necessary for the marginal limit itself. If that correction has a nonzero limit, its joint behavior with the projected transform must instead be taken into account.
\end{remark}

\begin{example}[Why off-axis zeros cannot simply be projected]\label{ex:zeta-off-axis}
Let $U$ be uniform on $[0,1]$, fix $a>0$, and consider the stationary off-axis divisor
\[
 \{n+U+\ii a:n\in\Z\}\ \cup\ \{n+U-\ii a:n\in\Z\}.
\]
Its projected measure has intensity two. At a fixed real $s$, its principal-value transform divided by $2\pi$ is $\cot(\pi(s-U))$, which is standard Cauchy. By contrast, the logarithmic derivative associated with the off-axis divisor, with the same normalization, is
\[
 \frac12\{\cot(\pi(s-U+\ii a))+\cot(\pi(s-U-\ii a))\}
 =\Re\cot(\pi(s-U+\ii a)).
\]
Its modulus is bounded by $1/\sinh(2\pi a)$. Thus stationarity and bounded counting discrepancy do not suffice to justify horizontal projection. Control of the off-axis contribution is required for that comparison.
\end{example}

\section{Function-field zeta functions}
\label{sec:functionfields}

The finite-circle identity has an immediate arithmetic application to polynomials whose reciprocal roots have the same absolute value.  It gives an exact law for each pure cohomological factor of the zeta function of a function field.  For the full zeta function, Poincar\'e duality allows us to pair the remaining factors and to estimate their contribution.  We first give the polynomial identity, then treat curves, and finally give the corresponding statement in arbitrary dimension.

\begin{theorem}[A polynomial with reciprocal roots on one circle]
\label{thm:pure}
Let $q>1$, let $d\geq1$ be an integer, and let $\sigma_0\in\R$.  Suppose that
\[
 P(u)=\prod_{j=1}^{d}(1-\alpha_j u),
 \qquad |\alpha_j|=q^{\sigma_0},
 \qquad \Lambda(s)=P(q^{-s}).
\]
If $\tau$ is uniform on $[0,2\pi/\log q)$, then
\begin{equation}
\label{eq:pure-cauchy}
 \frac{2\ii}{d\log q}\frac{\Lambda'}{\Lambda}
      (\sigma_0+\ii\tau)+\ii
 \ \sim\ \Cauchy(0,1).
\end{equation}
If the coefficients or reciprocal roots are random and independent of $\tau$, the assertion holds conditionally on them, and the variable in \eqref{eq:pure-cauchy} is independent of them.
\end{theorem}

\begin{proof}
Conditioning if necessary, we may assume that the reciprocal roots are fixed.  Write $\alpha_j=q^{\sigma_0}e^{\ii\phi_j}$.  Direct differentiation gives
\[
 \frac{\Lambda'}{\Lambda}(s)
   =\log q\sum_{j=1}^{d}
      \frac{\alpha_jq^{-s}}{1-\alpha_jq^{-s}}.
\]
For $\psi\notin2\pi\Z$,
\[
 \frac{e^{\ii\psi}}{1-e^{\ii\psi}}
 =-\frac12+\frac{\ii}{2}\cot\frac{\psi}{2}.
\]
Consequently, away from the finitely many zeros in one period,
\[
 \frac{2\ii}{d\log q}\frac{\Lambda'}{\Lambda}
       (\sigma_0+\ii\tau)+\ii
 =\frac1d\sum_{j=1}^{d}
       \cot\frac{\tau\log q-\phi_j}{2}.
\]
The angle $\tau\log q/2$ is uniform modulo $\pi$.  Proposition~\ref{prop:cotangent} proves the assertion.  Since the conditional law does not depend on the reciprocal roots, it also proves the stated independence.
\end{proof}

For the rest of this section, $q$ is a prime power. Let $X$ be a smooth projective variety of dimension $d$ over $\mathbf F_q$, let $\ell$ be a prime not dividing $q$, and put $X_{\overline{\mathbf F}_q}=X\times_{\mathbf F_q}\overline{\mathbf F}_q$. The $\ell$-adic étale cohomology
\[
 V_{X,j}=H^j_{\mathrm{\acute{e}t}}
       (X_{\overline{\mathbf F}_q},\mathbf Q_\ell),
\]
the $\ell$-adic analogue of singular cohomology, is a finite-dimensional $\mathbf Q_\ell$-vector space, zero unless $0\le j\le 2d$. The $q$-power Frobenius $F_q\colon X\to X$ ($f\mapsto f^q$ on functions) induces the geometric Frobenius $F_q^*$ on each $V_{X,j}$; the arithmetic Frobenius $x\mapsto x^q$ of $\operatorname{Gal}(\overline{\mathbf F}_q/\mathbf F_q)$ acts as its inverse. We use the geometric convention: the eigenvalues of $F_q^*$ on $V_{X,j}$ have absolute value $q^{j/2}$, not $q^{-j/2}$. Put
\begin{equation}
\label{eq:cohomological-factors}
 b_j=\dim_{\mathbf Q_\ell}V_{X,j},\qquad
 P_{X,j}(u)=\det\bigl(I-uF_q^*\mid V_{X,j}\bigr),
 \qquad \Lambda_{X,j}(s)=P_{X,j}(q^{-s}).
\end{equation}
Thus $b_j$ is the $j$-th Betti number of $X$, and if $\alpha_{j,1},\dots,\alpha_{j,b_j}$ are the eigenvalues of $F_q^*$ on $V_{X,j}$ then $P_{X,j}(u)=\prod_i(1-\alpha_{j,i}u)$, of degree $b_j$ with constant term $1$. By Deligne's theorem the $P_{X,j}$ have integer coefficients and are independent of $\ell$, so the $\alpha_{j,i}$ are algebraic integers, and $|\alpha_{j,i}|=q^{j/2}$ under every embedding into $\mathbf C$; we accordingly regard the $P_{X,j}$ as complex polynomials, with all roots on the circle $|u|=q^{-j/2}$. Under $u=q^{-s}$ each $P_{X,j}$ becomes the entire function $\Lambda_{X,j}$, whose zeros are the solutions of $q^{s}=\alpha_{j,i}$: $b_j$ vertical arithmetic progressions of step $2\pi i/\log q$, all on the line $\operatorname{Re}s=j/2$.

\begin{corollary}[Pure cohomological factors]
\label{cor:cohomology}
Let $X/\mathbf F_q$ be smooth and projective, and let $r\geq0$ satisfy $b_r>0$.  For $\tau$ uniform on $[0,2\pi/\log q)$,
\begin{equation}
\label{eq:cohomological-cauchy}
 \frac{2\ii}{b_r\log q}
 \frac{\Lambda_{X,r}'}{\Lambda_{X,r}}
       \left(\frac r2+\ii\tau\right)+\ii
 \ \sim\ \Cauchy(0,1).
\end{equation}
For a random variety over a fixed field $\mathbf F_q$, independent of $\tau$ and satisfying $b_r>0$ almost surely, the assertion holds conditionally on $X$, and the variable in \eqref{eq:cohomological-cauchy} is independent of $X$.
\end{corollary}

\begin{proof}
Deligne's purity theorem \cite[Th\'eor\`eme~(1.6)]{Deligne} gives both the assertion about the coefficients of $P_{X,r}$ and the equality $|\alpha|=q^{r/2}$ for every complex reciprocal root.  Apply Theorem~\ref{thm:pure} with $d=b_r$ and $\sigma_0=r/2$, conditionally on $X$ in the random case.
\end{proof}

If $X$ is of pure dimension $n$, its full zeta function, in the variable $s$, is
\begin{equation}
\label{eq:full-cohomological-factorization}
 Z_X(s)=\prod_{j=0}^{2n}\Lambda_{X,j}(s)^{(-1)^{j+1}}.
\end{equation}
The logarithmic derivative is therefore a signed sum of contributions from different purity lines.  The exact law of a single factor does not determine the law of this sum.  For curves, the remaining factors are explicit; in higher dimension, we shall use their duality in addition to purity.

Let $C/\mathbf F_q$ be a smooth, projective, geometrically connected curve of genus $g\geq1$.  Write
\begin{equation}
\label{eq:curve-zeta-factorization}
 Z_C(s)=\frac{P_C(q^{-s})}{(1-q^{-s})(1-q^{1-s})},
 \qquad
 P_C(u)=\prod_{j=1}^{2g}(1-\alpha_j u),
 \qquad \Lambda_C(s)=P_C(q^{-s}).
\end{equation}
Weil's theorem gives $|\alpha_j|=\sqrt q$; see \cite{Weil}.  We denote the Kolmogorov distance between probability measures on $\R$ by
\[
 \dK(\mu,\nu)
   =\sup_{x\in\R}
       |\mu(({-\infty},x])-\nu(({-\infty},x])|.
\]

\begin{corollary}[Curves over finite fields]
\label{cor:curves}
For $\tau$ uniform on $[0,2\pi/\log q)$, the variable
\begin{equation}
\label{eq:curve-numerator-cauchy}
 C_C=\frac{\ii}{g\log q}\frac{\Lambda_C'}{\Lambda_C}
           \left(\frac12+\ii\tau\right)+\ii
\end{equation}
is exactly standard Cauchy.  Define
\begin{equation}
\label{eq:curve-full-variable}
 Y_C=\frac{\ii}{g\log q}\frac{Z_C'}{Z_C}
           \left(\frac12+\ii\tau\right)
          +\ii\left(1-\frac1g\right),
 \qquad
 \delta_{q,g}=\frac{2\sqrt q}{g(q-1)}.
\end{equation}
Then $Y_C$ is real almost surely, and the coupling through the same $\tau$ satisfies $|Y_C-C_C|\leq\delta_{q,g}$ almost surely.  Moreover,
\begin{align}
\label{eq:curve-bl-bound}
 \dBL\bigl(\Law(Y_C),\Cauchy(0,1)\bigr)
 &\leq \frac{2}{\pi g}
          \log\frac{\sqrt q+1}{\sqrt q-1}
 \leq\delta_{q,g}
 \leq\frac{2\sqrt2}{g},
 \\
\label{eq:curve-kolmogorov-bound}
 \dK\bigl(\Law(Y_C),\Cauchy(0,1)\bigr)
 &\leq\frac2\pi\arctan\frac{\delta_{q,g}}2
 \leq\frac{2\sqrt q}{\pi g(q-1)}
 \leq\frac{2\sqrt2}{\pi g}.
\end{align}
Thus $Y_C$ converges in distribution to a standard Cauchy variable along every sequence of curves for which $g(q-1)/\sqrt q\to\infty$.  In particular, this holds as $g\to\infty$, uniformly in $q$, and as $q\to\infty$ with fixed positive genus.
\end{corollary}

\begin{proof}
The exact law follows from Theorem~\ref{thm:pure} with $d=2g$ and $\sigma_0=1/2$.  Logarithmic differentiation of \eqref{eq:curve-zeta-factorization} gives
\[
 \frac1{\log q}
 \left(\frac{Z_C'}{Z_C}(s)-\frac{\Lambda_C'}{\Lambda_C}(s)\right)
 =-\frac{q^{-s}}{1-q^{-s}}-\frac{q^{1-s}}{1-q^{1-s}}.
\]
On the line $s=1/2+\ii\tau$, put $a=q^{-1/2}e^{-\ii\tau\log q}$ and $c=(1-a)^{-1}$.  Since $q^{1-s}=1/\overline a$, the right-hand side equals
\[
 1-c+\overline c=1-2\ii\Im c.
\]
The center in \eqref{eq:curve-full-variable} therefore yields the exact real correction
\begin{equation}
\label{eq:curve-real-correction}
 Y_C-C_C=\frac2g\Im c.
\end{equation}

For $0<r<1$, elementary maximization gives
\begin{equation}
\label{eq:circle-imaginary-bound}
 \sup_{\theta\in\R}
 \left|\Im\frac1{1-re^{-\ii\theta}}\right|
 =\sup_{\theta\in\R}
   \frac{r|\sin\theta|}{1-2r\cos\theta+r^2}
 =\frac r{1-r^2}.
\end{equation}
Indeed, on $[0,\pi]$ the derivative of the middle expression has the sign of $(1+r^2)\cos\theta-2r$, and substitution of its unique zero gives the last expression.  For $r=q^{-1/2}$, ${\|Y_C-C_C\|_{L^\infty}=\delta_{q,g}}$ with respect to $\tau$.  Integration gives
\begin{align*}
 \E|Y_C-C_C|
 &=\frac{2r}{\pi g}\int_0^\pi
       \frac{\sin\theta}{1-2r\cos\theta+r^2}\,\dd\theta
 =\frac{2}{\pi g}\log\frac{1+r}{1-r}.
\end{align*}
These are exact costs of this coupling; no optimality for $\dBL$ or $\dK$ is asserted.  The defining test functions for $\dBL$ give the first inequality in \eqref{eq:curve-bl-bound}; the remaining inequalities follow from the almost-sure bound and $q\geq2$.

Let $F(x)=1/2+\pi^{-1}\arctan x$ be the standard Cauchy distribution function.  The coupling implies
\[
 F(x-\delta_{q,g})\leq\P(Y_C\leq x)
       \leq F(x+\delta_{q,g}).
\]
For every $\delta\geq0$,
\begin{equation}
\label{eq:cauchy-continuity-modulus}
 \sup_x\{F(x+\delta)-F(x)\}
 =\frac2\pi\arctan\frac\delta2
 \leq\frac\delta\pi.
\end{equation}
The maximum is attained at $x=-\delta/2$.  This proves \eqref{eq:curve-kolmogorov-bound} and the convergence assertions.
\end{proof}

\begin{remark}
At fixed genus the exact statement concerns $\Lambda_C$, while the full zeta function is covered by the explicit approximation above.  The vertical density of the zeros of a degree-$d$ factor in Theorem~\ref{thm:pure} is $d\log q/(2\pi)$, counted with multiplicity.  Thus the coefficient of its logarithmic derivative in \eqref{eq:pure-cauchy} is $\ii$ divided by $\pi$ times this density, as in the stationary and number-field normalizations.
\end{remark}

We now extend the curve comparison to the full zeta function of a variety.  The central factor remains exactly Cauchy, and the factors on either side of the central line can be paired.  The hypothesis $b_n>0$ is automatic in even dimension, since the class of an ample divisor has a nonzero $n/2$-th power.  In odd dimension the middle factor can be absent.

\begin{theorem}[Full zeta functions in arbitrary dimension]
\label{thm:full-variety}
Let $X/\mathbf F_q$ be smooth, projective and geometrically connected, of pure dimension $n\geq1$, and assume $b_n>0$.  Let $\tau$ be uniform on $[0,2\pi/\log q)$ and, with the notation of \eqref{eq:cohomological-factors}, put
\[
 \chi_X=\sum_{j=0}^{2n}(-1)^j b_j,
 \qquad \varepsilon=(-1)^{n+1},
 \qquad s_\tau=\frac n2+\ii\tau.
\]
Define
\begin{align}
\label{eq:full-variety-variable}
 Y_X&=\frac{2\ii\varepsilon}{b_n\log q}
          \frac{Z_X'}{Z_X}(s_\tau)
          -\frac{\ii\varepsilon\chi_X}{b_n},
 \\
\label{eq:full-variety-error}
 \Delta_X&=\frac4{b_n}
      \sum_{j=0}^{n-1}
       \frac{b_jq^{(n-j)/2}}{q^{n-j}-1}.
\end{align}
Then $Y_X$ is real almost surely with respect to $\tau$.  Under the coupling through $\tau$, the standard Cauchy variable
\begin{equation}
\label{eq:middle-factor-variable}
 C_X=\frac{2\ii}{b_n\log q}
          \frac{\Lambda_{X,n}'}{\Lambda_{X,n}}(s_\tau)+\ii
\end{equation}
satisfies $|Y_X-C_X|\leq\Delta_X$ almost surely.  In particular,
\begin{align}
\label{eq:full-variety-distances}
 \dBL\bigl(\Law(Y_X),\Cauchy(0,1)\bigr)&\leq\Delta_X,
 &
 \dK\bigl(\Law(Y_X),\Cauchy(0,1)\bigr)
 &\leq\frac2\pi\arctan\frac{\Delta_X}{2}.
\end{align}
For any sequence of such varieties, $\Delta_X\to0$ implies $Y_X\Rightarrow\Cauchy(0,1)$.  In particular, this holds if
\begin{equation}
\label{eq:middle-betti-dominance}
 \frac1{b_n}\sum_{j<n}b_j\longrightarrow0,
\end{equation}
uniformly in the cardinality of the base field.
\end{theorem}

\begin{proof}
Let $\varepsilon_j=(-1)^{j+1}$ and
\[
 L_j(s)=\frac1{\log q}
             \frac{\Lambda_{X,j}'}{\Lambda_{X,j}}(s).
\]
Poincar\'e duality gives $b_{2n-j}=b_j$ and pairs the eigenvalues on $V_{X,j}$ with $q^n/\alpha$ on $V_{X,2n-j}$; see, for example, the proof of \cite[Theorem~27.12]{Milne}.  Moreover, the reciprocal roots of $P_{X,j}$ form a multiset stable under complex conjugation, because $P_{X,j}$ has integer coefficients.  Thus the eigenvalues on $V_{X,2n-j}$ can also be listed as $q^n/\overline\alpha$, with $\alpha$ running over the eigenvalues on $V_{X,j}$.

Fix $j<n$, and write these latter eigenvalues as $\alpha_{j,k}$, $1\leq k\leq b_j$.  At $s=s_\tau$, set
\[
 a_{j,k}=\alpha_{j,k}q^{-s_\tau},
 \qquad c_{j,k}=\frac1{1-a_{j,k}},
 \qquad r_j=q^{-(n-j)/2}<1.
\]
Purity gives $|a_{j,k}|=r_j$.  The scaled eigenvalue from the paired degree is
\[
 \frac{q^n}{\overline{\alpha_{j,k}}}q^{-s_\tau}
       =\frac1{\overline{a_{j,k}}}.
\]
For each such pair, its contribution to $L_j+L_{2n-j}$ is
\[
 \frac{a_{j,k}}{1-a_{j,k}}
 +\frac{1/\overline{a_{j,k}}}{1-1/\overline{a_{j,k}}}
 =c_{j,k}-1-\overline{c_{j,k}}
 =-1+2\ii\Im c_{j,k}.
\]
Since $\varepsilon_{2n-j}=\varepsilon_j$, logarithmic differentiation of \eqref{eq:full-cohomological-factorization} yields
\begin{equation}
\label{eq:paired-full-logarithmic-derivative}
 \frac1{\log q}\frac{Z_X'}{Z_X}(s_\tau)
 =\varepsilon L_n(s_\tau)
   -\sum_{j<n}\varepsilon_j b_j
   +2\ii\sum_{j<n}\varepsilon_j
             \sum_{k=1}^{b_j}\Im c_{j,k}.
\end{equation}
The Euler characteristic can be written as
\[
 \chi_X=-\varepsilon b_n-2\sum_{j<n}\varepsilon_j b_j.
\]
Substitution into \eqref{eq:full-variety-variable} cancels all the constant imaginary terms and gives
\begin{equation}
\label{eq:full-variety-real-correction}
 Y_X-C_X
 =-\frac{4\varepsilon}{b_n}
       \sum_{j<n}\varepsilon_j
             \sum_{k=1}^{b_j}\Im c_{j,k}.
\end{equation}
Corollary~\ref{cor:cohomology} gives the exact standard Cauchy law of $C_X$.  The right-hand side of \eqref{eq:full-variety-real-correction} is real, and \eqref{eq:circle-imaginary-bound} bounds its absolute value by
\[
 \frac4{b_n}\sum_{j<n}\frac{b_jr_j}{1-r_j^2}
 =\Delta_X.
\]
The possible singularities in this argument occur only at the finitely many zeros of the middle factor in one period; they form a null event under the uniform choice of $\tau$.

The coupling bound for $\dBL$ and \eqref{eq:cauchy-continuity-modulus} prove \eqref{eq:full-variety-distances}.  Finally, for $q\geq2$ and every integer $k\geq1$,
\[
 \frac{q^{k/2}}{q^k-1}
 =\frac{q^{-k/2}}{1-q^{-k}}
 \leq\sqrt2.
\]
It follows that $\Delta_X\leq4\sqrt2\,b_n^{-1}\sum_{j<n}b_j$, which proves the last assertion.
\end{proof}

\begin{remark}
For $n=1$, one has $b_1=2g$ and $b_0=1$, so $\Delta_X=2\sqrt q/[g(q-1)]=\delta_{q,g}$, recovering the curve coupling bound.  Theorem~\ref{thm:full-variety} also applies conditionally to random varieties over a fixed field when the uniform angle is chosen independently.  For a fixed $X/\mathbf F_q$, extension of the base field to $\mathbf F_{q^m}$ leaves the Betti numbers unchanged, and \eqref{eq:full-variety-error}, with $q$ replaced by $q^m$, tends to zero as $m\to\infty$.
\end{remark}

\begin{corollary}[Smooth hypersurfaces]
\label{cor:hypersurfaces}
Fix $n\geq1$.  Let $X/\mathbf F_q$ be a smooth hypersurface of degree $d$ in $\mathbb P^{n+1}$, and suppose that
\begin{equation}
\label{eq:hypersurface-middle-betti}
 b_n(d)=\mathbf1_{\{n\text{ even}\}}
       +\frac{(d-1)^{n+2}+(-1)^n(d-1)}d>0.
\end{equation}
This condition fails precisely when $n$ is odd and $d\in\{1,2\}$, corresponding to hyperplanes and smooth quadrics.
For the variable $Y_X$ of Theorem~\ref{thm:full-variety},
\begin{equation}
\label{eq:hypersurface-error}
 \Delta_X
 =\frac4{b_n(d)}
   \sum_{\substack{0\leq j<n\\j\text{ even}}}
      \frac{q^{(n-j)/2}}{q^{n-j}-1}
 \leq\frac{4\sqrt2\,\lceil n/2\rceil}{b_n(d)}.
\end{equation}
In particular, as $d\to\infty$ through smooth hypersurfaces of dimension $n$, the law of $Y_X$ converges to $\Cauchy(0,1)$ uniformly in $q$ and in the hypersurface.  Both distances in \eqref{eq:full-variety-distances} are $O_n(d^{-(n+1)})$.
\end{corollary}

\begin{proof}
A smooth projective hypersurface of positive dimension is geometrically connected.  Weak Lefschetz and Poincar\'e duality give
\[
 b_j=\begin{cases}1,&j\text{ even},\\0,&j\text{ odd},\end{cases}
 \qquad j\ne n.
\]
The middle Betti number is given by \eqref{eq:hypersurface-middle-betti}.  These statements hold over every finite field; see \cite[Proposition~3.1 and Example~4.3(i)]{GhorpadeLachaud}, where the primitive middle Betti number is computed, and the full middle Betti number is obtained by adding $1$ when $n$ is even.  Applying Theorem~\ref{thm:full-variety} gives the equality in \eqref{eq:hypersurface-error}.  There are $\lceil n/2\rceil$ even integers in $[0,n)$, so the last inequality follows from the bound used at the end of its proof.  Finally, $b_n(d)\sim d^{n+1}$ for fixed $n$, proving the asserted rate.
\end{proof}

\begin{example}
For smooth degree-$d$ surfaces in $\mathbb P^3$, one has
\[
 b_2=d^3-4d^2+6d-2,\qquad \chi_X=b_2+2.
\]
Thus the normalized full zeta derivative
\[
 -\frac{2\ii}{b_2\log q}\frac{Z_X'}{Z_X}(1+\ii\tau)
       +\ii\left(1+\frac2{b_2}\right)
\]
can be coupled to a standard Cauchy variable with error at most $\Delta_X=\frac{4q}{b_2(q^2-1)}$, almost surely. At any fixed $q$, smooth hypersurfaces of unbounded degree are supplied, for example, by the Fermat equations
\[
 x_0^d+\cdots+x_{n+1}^d=0,
 \qquad \operatorname{char}(\mathbf F_q)\nmid d.
\]
Their partial derivatives do not vanish simultaneously at any projective point, which verifies smoothness and gives explicit sequences to which Corollary~\ref{cor:hypersurfaces} applies.
\end{example}

\appendix
\section{The uniform local Hadamard factorization}\label{app:factorization}

We give the factorization and uniform error estimate used in Lemma~\ref{lem:local-log}. Let $N(v)$ count all nontrivial zeros with ordinate in $(0,v]$, with multiplicity, and extend it oddly to $v<0$. The Riemann--von Mangoldt formula gives
\begin{equation}\label{eq:appendix-counting}
 N(v)=\Phi_0(v)+U(v),\qquad
 \Phi_0'(v)=\frac1{2\pi}\log\left|\frac v{2\pi}\right|,
 \qquad U(v)=O(\log(2+|v|)),
\end{equation}
where $\Phi_0(0)=0$ and its derivative is interpreted almost everywhere. The difference between $\Phi_0$ and $\Phi$ in \eqref{eq:zeta-counting-formula} is bounded. We use the right-continuous representative of the signed counting function in Stieltjes integrals. Changing conventions at cutoff ordinates affects only the endpoint error estimates below. No Riemann hypothesis is used. All sums and products run over distinct zeros, with multiplicity denoted by $m(\rho)$.

\begin{proposition}\label{prop:local-factorization}
Put $L=\log T$ and $s_t=1/2+\ii t$. For a nontrivial zero $\rho=\beta+\ii\gamma$, put
\[
 u_{\rho,t,T}=\frac{(\gamma-t)L}{2\pi},
 \qquad
 z_{\rho,t,T}=\frac{L}{2\pi\ii}(\rho-s_t).
\]
For every fixed $r\ge0$ and all sufficiently large $T$, there are functions $\varepsilon_{t,T}$ holomorphic on $|w|<r+1$ such that, uniformly for $t\in[T/L,T]$ with $\zeta(s_t)\ne0$ and $|w|\le r$,
\begin{equation}\label{eq:local-factorization}
 \frac{\zeta(s_t-2\pi\ii w/L)}{\zeta(s_t)}
 =e^{\ii\pi w}(1+\varepsilon_{t,T}(w))
 \prod_{|u_{\rho,t,T}|\le L^2}
 \left(1+\frac w{z_{\rho,t,T}}\right)^{m(\rho)},
\end{equation}
and
\begin{equation}\label{eq:local-factorization-error}
 \sup_{\substack{t\in[T/L,T]:\zeta(s_t)\ne0\\|w|\le r}}
 |\varepsilon_{t,T}(w)|
 =O_r\left(\frac{\log L}{L}\right).
\end{equation}
More precisely, for each admissible $t$,
\begin{equation}\label{eq:local-factorization-refined}
 \sup_{|w|\le r}|\varepsilon_{t,T}(w)|
 \le C_r\frac{1+\log(T/t)}L.
\end{equation}
In particular,
\begin{equation}\label{eq:local-log-uniform}
 \frac{2\ii}{L}\frac{\zeta'}{\zeta}(s_t)+\ii
 =-\frac1\pi\sum_{|u_{\rho,t,T}|\le L^2}
 \frac{m(\rho)}{z_{\rho,t,T}}
 +O\left(\frac{\log L}{L}\right)
\end{equation}
uniformly for the same heights. Its remainder is also bounded by $C(1+\log(T/t))/L$.
\end{proposition}

\begin{proof}
We first control the symmetric tails of the product for the completed function
\[
 \xi(s)=\frac12s(s-1)\pi^{-s/2}\Gamma(s/2)\zeta(s).
\]
Its symmetric Hadamard product is
\begin{equation}\label{eq:xi-hadamard}
 \xi(s)=\xi(0)\lim_{Y\to\infty}
 \prod_{|\Im\rho|\le Y}\left(1-\frac s\rho\right)^{m(\rho)};
\end{equation}
see \cite[Chapter~II]{Titchmarsh}. Dividing the products at $s_t-2\pi\ii w/L$ and $s_t$ gives
\begin{equation}\label{eq:xi-ratio}
 \frac{\xi(s_t-2\pi\ii w/L)}{\xi(s_t)}
 =\lim_{B\to\infty}\prod_{|u_{\rho,t,T}|\le B}
 \left(1+\frac w{z_{\rho,t,T}}\right)^{m(\rho)}.
\end{equation}
To justify the change of center of the symmetric cutoff, keep $t,T$ fixed. At height $Y$, the difference of the two orderings is contained in intervals of bounded length $O(1+t)$ adjacent to $Y$ and $-Y$. They contain $O((1+t)\log(2+Y+t))$ zeros by \eqref{eq:appendix-counting}, and each logarithm of a ratio factor is $O_r((LY)^{-1})$. Their total contribution tends to zero. This proves \eqref{eq:xi-ratio}, locally uniformly in $w$.

For the rest of the proof abbreviate $u=u_{\rho,t,T}$ and $z=z_{\rho,t,T}$. Let $A\ge L+2(r+1)$ and $B>A$. We first estimate the projected reciprocal sum. Stieltjes integration of \eqref{eq:appendix-counting} over the two tails gives
\begin{equation}\label{eq:reciprocal-tail}
 \sum_{A<|u|\le B}\frac{m(\rho)}u
 =\frac1L\int_{2\pi A/(tL)}^{2\pi B/(tL)}
 \log\left|\frac{1+y}{1-y}\right|\frac{\dd y}{y}
 +O\left(\frac{L+\log(2+A)}A\right).
\end{equation}
Indeed, the smooth part is
\[
 \frac1L\int_{2\pi A/L}^{2\pi B/L}
 \frac{\log(t+v)-\log|t-v|}{v}\dd v,
\]
which is the integral displayed in \eqref{eq:reciprocal-tail}. For the error part, the bound
\[
 |U(t\pm2\pi x/L)|\le C(L+\log(2+x))
\]
and integration by parts against $1/x$ bound the endpoint terms and the remaining integral by $C(L+\log(2+A))/A$, uniformly in $B$. The integral
\[
 \int_0^\infty\log\left|\frac{1+y}{1-y}\right|\frac{\dd y}{y}
\]
is finite: its integrand is bounded at zero, has an integrable logarithmic singularity at one, and is $O(y^{-2})$ at infinity. Therefore the reciprocal sum in \eqref{eq:reciprocal-tail} is
\begin{equation}\label{eq:reciprocal-tail-bound}
 O\left(\frac{L+\log(2+A)}A+\frac1L\right).
\end{equation}

A dyadic application of \eqref{eq:appendix-counting} also gives
\begin{equation}\label{eq:square-tail-bound}
 \sum_{|u|>A}\frac{m(\rho)}{u^2}
 \le C\left\{\frac1A\left(1+\frac{\log(2+A)}L\right)
                 +\frac L{A^2}\right\},\qquad A\ge L.
\end{equation}
To see this, the number of zeros in $X<|u|\le2X$ is at most
\[
 C\left\{X\left(1+\frac{\log(2+X)}L\right)+L\right\}.
\]
The term $X\log(2+X)/L$ accounts for ordinates far above $T$. Dividing by $X^2$ and summing over $X=2^kA$ proves the bound.

For $|w|\le r$, Taylor's estimate
\[
 \log(1+w/u)=w/u+O_r(u^{-2})
\]
is uniform on these tails. Equations \eqref{eq:reciprocal-tail-bound}--\eqref{eq:square-tail-bound} therefore bound the logarithm of the projected product. To remove the projection, note that $|\Im z|\le L/(4\pi)$ and $\Re z=u$. Along the vertical segment joining $u$ to $z$, the derivative of $v\mapsto\log(1+w/v)$ has absolute value at most $C_r/u^2$. Hence
\[
 |\log(1+w/z)-\log(1+w/u)|\le C_rL/u^2.
\]
Combining the estimates yields
\begin{equation}\label{eq:complex-product-tail}
 \sup_{|w|\le r}
 \left|\sum_{A<|u|\le B}m(\rho)\log(1+w/z)\right|
 \le C_r\left\{\frac{L+\log(2+A)}A+\frac1L+\frac{L^2}{A^2}\right\}.
\end{equation}
The symmetric reciprocal sum converges as $B\to\infty$: the smooth integral in \eqref{eq:reciprocal-tail} converges, the endpoint terms $U(t\pm2\pi B/L)/B$ tend to zero, and the error integrals converge absolutely because $(L+\log(2+x))/x^2$ is integrable at infinity. All the quadratic remainders converge absolutely as well. Thus the logarithmic tails converge locally uniformly and define a nonvanishing holomorphic factor. Taking $A=L^2$ in \eqref{eq:complex-product-tail} gives
\begin{equation}\label{eq:xi-local-product}
 \frac{\xi(s_t-2\pi\ii w/L)}{\xi(s_t)}
 =(1+O_r(L^{-1}))
 \prod_{|u|\le L^2}(1+w/z)^{m(\rho)},
\end{equation}
where the multiplicative error is holomorphic and uniform in $t$.

Write $\xi(s)=A(s)\zeta(s)$, with
\[
 A(s)=\tfrac12s(s-1)\pi^{-s/2}\Gamma(s/2).
\]
Stirling's formula gives, uniformly on bounded $w$-disks and $t\in[T/L,T]$,
\begin{align*}
 \log\frac{A(s_t)}{A(s_t-2\pi\ii w/L)}
 &=\frac{2\pi\ii w}{L}
     \left\{\frac12\log\frac t{2\pi}+O(1)\right\}
       +O_r((tL^2)^{-1})\\
 &=\ii\pi w+O_r\left(\frac{1+\log(T/t)}L\right),
\end{align*}
since $\log(t/(2\pi))-L=-\log(T/t)-\log(2\pi)$. Multiplying \eqref{eq:xi-local-product} by this factor proves \eqref{eq:local-factorization} and \eqref{eq:local-factorization-refined}; the uniform estimate \eqref{eq:local-factorization-error} follows from $\log(T/t)\le\log L$. The construction on a slightly larger disk supplies the holomorphic extension, and $\varepsilon_{t,T}(0)=0$. Applying \eqref{eq:local-factorization-refined} with $r=1$, Cauchy's estimate gives $|\varepsilon'_{t,T}(0)|\le C(1+\log(T/t))/L$. Differentiating \eqref{eq:local-factorization} at zero yields
\[
 -\frac{2\pi\ii}{L}\frac{\zeta'}{\zeta}(s_t)
 =\ii\pi+\sum_{|u|\le L^2}\frac{m(\rho)}z
       +\varepsilon'_{t,T}(0),
\]
so the remainder in \eqref{eq:local-log-uniform} is $-\varepsilon'_{t,T}(0)/\pi$. This proves both assertions about it.
\end{proof}

\section*{Acknowledgements}
The authors thank Sasha Sodin for explaining the connection between Stieltjes transforms of sine processes and logarithmic derivatives, and for sharing ideas concerning the zeta-function application.


\begin{thebibliography}{99}

\bibitem{AW}
M.~Aizenman and S.~Warzel,
\newblock On the ubiquity of the Cauchy distribution in spectral problems,
\newblock \emph{Probab. Theory Related Fields} \textbf{163} (2015), 61--87.

\bibitem{Assiotis}
T.~Assiotis,
\newblock Random entire functions from random polynomials with real zeros,
\newblock \emph{Adv. Math.} \textbf{410} (2022), 108701.

\bibitem{BahrEsseen}
B.~von Bahr and C.-G. Esseen,
\newblock Inequalities for the $r$-th absolute moment of a sum of random
variables, $1\leq r\leq2$,
\newblock \emph{Ann. Math. Statist.} \textbf{36} (1965), 299--303.

\bibitem{Boole}
G.~Boole,
\newblock On the comparison of transcendents, with certain applications to
the theory of definite integrals,
\newblock \emph{Philos. Trans. Roy. Soc. London} \textbf{147} (1857),
745--803.
\newblock \href{https://doi.org/10.1098/rstl.1857.0037}
{doi:10.1098/rstl.1857.0037}.

\bibitem{CNN}
R.~Chhaibi, J.~Najnudel and A.~Nikeghbali,
\newblock The circular unitary ensemble and the Riemann zeta function:
the microscopic landscape and a new approach to ratios,
\newblock \emph{Invent. Math.} \textbf{207} (2017), 23--113.
\newblock doi:\href{https://doi.org/10.1007/s00222-016-0669-1}{10.1007/s00222-016-0669-1}.

\bibitem{Deligne}
P.~Deligne,
\newblock La conjecture de Weil. I,
\newblock \emph{Publ. Math. Inst. Hautes \'Etudes Sci.} \textbf{43} (1974),
273--307.
\newblock \href{https://doi.org/10.1007/BF02684373}
{doi:10.1007/BF02684373}.

\bibitem{Forrester}
P.~J. Forrester,
\newblock Joint moments of a characteristic polynomial and its derivative
for the circular $\beta$-ensemble,
\newblock \emph{Probab. Math. Phys.} \textbf{3} (2022), 145--170.

\bibitem{Fujii1974}
A.~Fujii,
\newblock On the zeros of Dirichlet $L$-functions. I,
\newblock \emph{Trans. Amer. Math. Soc.} \textbf{196} (1974), 225--235.

\bibitem{Fujii1981}
A.~Fujii,
\newblock On the zeros of Dirichlet $L$-functions. II (with corrections to ``On the zeros of Dirichlet $L$-functions. I'' and the subsequent papers),
\newblock \emph{Trans. Amer. Math. Soc.} \textbf{267} (1981), 33--40.

\bibitem{Fujii1999}
A.~Fujii,
\newblock Explicit formulas and oscillations,
\newblock in \emph{Emerging Applications of Number Theory}, D.~Hejhal, J.~Friedman, M.~Gutzwiller and A.~Odlyzko (eds.), Springer, 1999, 219--267.

\bibitem{GhorpadeLachaud}
S.~R. Ghorpade and G.~Lachaud,
\newblock \'Etale cohomology, Lefschetz theorems and number of points of singular varieties over finite fields,
\newblock \emph{Mosc. Math. J.} \textbf{2} (2002), no.~3, 589--631;
Corrigenda and Addenda, \textbf{9} (2009), no.~2, 431--438.
\newblock Corrected and revised version: \href{https://arxiv.org/abs/0808.2169}{arXiv:0808.2169}.

\bibitem{Kallenberg}
O.~Kallenberg,
\newblock \emph{Random Measures, Theory and Applications},
\newblock Springer, Cham, 2017.

\bibitem{MaplesRodgers}
K.~Maples and B.~Rodgers,
\newblock Bootstrapped zero density estimates and a central limit theorem for the zeros of the zeta function,
\newblock \emph{Int. J. Number Theory} \textbf{11} (2015), 2087--2107.
\newblock \href{https://doi.org/10.1142/S1793042115500918}{doi:10.1142/S1793042115500918}.

\bibitem{Milne}
J.~S. Milne,
\newblock \emph{Lectures on \'Etale Cohomology}, version~2.21, 2013.
\newblock Available at \url{https://www.jmilne.org/math/CourseNotes/LEC.pdf}.

\bibitem{Companion}
J.~Najnudel and A.~Nikeghbali,
\newblock Random entire functions from their zeros: characteristic
polynomials and the stochastic zeta function,
\newblock Working preprint, 2026.

\bibitem{Mesoscopic}
J.~Najnudel and A.~Nikeghbali,
\newblock Holomorphic coloured-noise limits for mesoscopic resolvents,
\newblock Working preprint, 2026.

\bibitem{Bead}
J.~Najnudel and B.~Vir\'ag,
\newblock The bead process for beta ensembles,
\newblock \emph{Probab. Theory Related Fields} \textbf{179} (2021), 589--647.

\bibitem{Variance}
J.~Najnudel and B.~Vir\'ag,
\newblock Uniform point variance bounds in classical beta ensembles,
\newblock \emph{Random Matrices Theory Appl.} \textbf{10} (2021),
no.~4, 2150033.

\bibitem{PillaiMeng}
N.~S. Pillai and X.-L. Meng,
\newblock An unexpected encounter with Cauchy and L\'evy,
\newblock \emph{Ann. Statist.} \textbf{44} (2016), no.~5, 2089--2097.
\newblock \href{https://doi.org/10.1214/15-AOS1407}{doi:10.1214/15-AOS1407}.

\bibitem{PitmanWilliams}
E.~J.~G. Pitman and E.~J. Williams,
\newblock Cauchy-distributed functions of Cauchy variates,
\newblock \emph{Ann. Math. Statist.} \textbf{38} (1967), no.~3, 916--918.
\newblock \href{https://doi.org/10.1214/aoms/1177698885}{doi:10.1214/aoms/1177698885}.

\bibitem{Selberg}
A.~Selberg,
\newblock Contributions to the theory of the Riemann zeta-function,
\newblock \emph{Arch. Math. Naturvid.} \textbf{48} (1946), no.~5, 89--155.

\bibitem{Sodin}
S.~Sodin,
\newblock On the critical points of random matrix characteristic polynomials
and of the Riemann $\xi$-function,
\newblock \emph{Q. J. Math.} \textbf{69} (2018), 183--210.

\bibitem{Soshnikov}
A.~Soshnikov,
\newblock Determinantal random point fields,
\newblock \emph{Russian Math. Surveys} \textbf{55} (2000), no.~5, 923--975.

\bibitem{Stein}
E.~M. Stein,
\newblock \emph{Singular Integrals and Differentiability Properties of Functions},
\newblock Princeton Mathematical Series~30, Princeton University Press, Princeton, NJ, 1970.

\bibitem{Titchmarsh}
E.~C. Titchmarsh,
\newblock \emph{The Theory of the Riemann Zeta-function},
\newblock second edition, revised by D.~R. Heath-Brown, Clarendon Press, 1986.

\bibitem{ValkoVirag}
B.~Valk\'o and B.~Vir\'ag,
\newblock The many faces of the stochastic zeta function,
\newblock \emph{Geom. Funct. Anal.} \textbf{32} (2022), 1160--1231.

\bibitem{Weil}
A.~Weil,
\newblock \emph{Sur les courbes alg\'ebriques et les vari\'et\'es qui s'en
d\'eduisent},
\newblock Hermann, Paris, 1948.

\bibitem{Williams}
E.~J. Williams,
\newblock Cauchy-distributed functions and a characterization of the Cauchy distribution,
\newblock \emph{Ann. Math. Statist.} \textbf{40} (1969), no.~3, 1083--1085.
\newblock \href{https://doi.org/10.1214/aoms/1177697613}{doi:10.1214/aoms/1177697613}.

\end{thebibliography}
\end{document}